\documentclass[twoside,letterpaper,draft,11pt]{amsart}
\usepackage{amsmath,amsthm,latexsym,xspace,amscd,amssymb,mathtools,color,enumerate}

\DeclareUnicodeCharacter{00B4}{}

\usepackage[utf8]{inputenc}

\usepackage[mathscr]{eucal}
\usepackage{amssymb}
\usepackage[all]{xy}
\usepackage{amssymb,amsfonts,amsmath,amsthm,amscd}
\usepackage{relsize}
\usepackage{color}
\usepackage[english]{babel}
\makeindex

\makeatletter\renewcommand\theenumi{\@roman\c@enumi}\makeatother

\newtheorem{thm}{Theorem}[section]
\newtheorem{prop}[thm]{Proposition}
\newtheorem{lem}[thm]{Lemma}
\newtheorem{cor}[thm]{Corollary}
\theoremstyle{definition}
\newtheorem{defi}[thm]{Definition}
\newtheorem{exa}[thm]{Example}
\newtheorem{rem}[thm]{Remark}

\DeclareMathOperator{\Ima}{im}

\newcommand{\cC}{\mathcal{C}}

\newcommand{\gm}[1]{\operatorname{END}_{A}(#1)}
\newcommand{\indf}{\mathtt{Ind_R^A}\,}

\newcommand{\m}{{}^{-1}}

\newcommand{\ep}{\epsilon}

\newcommand{\U}{\mathcal{U}}
\newcommand{\Z}{\mathbb{Z}}
\newcommand{\N}{\mathbb{N}}
\newcommand{\af}{\alpha}
\newcommand{\afg}{\alpha_g}
\newcommand{\afh}{\alpha_h}
\newcommand{\afgh}{\alpha_{gh}}
\newcommand{\omgh}{\omega_{g,h}}
\newcommand{\om}{\omega}

\newcommand{\vm}{\vspace{.05cm}}

\newcommand{\id}{\operatorname{id}}
\newcommand{\sd}{\scriptscriptstyle{\rm sd}}

\newcommand{\lmod}[1]{\hspace{-2pt}{}_{#1}\mathcal{M}}

\newcommand{\pic}[1]{\mathtt{PicS}(#1)}

\newcommand{\compo}[1]{\mathtt{Comp}(#1)}
\newcommand{\picm}[1]{\mathtt{PicS_{R^\gamma}}(#1)}
\begin{document}
	
	\thispagestyle{empty}
	
	\title[Graded equivalences]{Epsilon-strongly graded rings: Azumaya algebras and partial crossed products}
	
	\author[D. Bagio]{Dirceu Bagio}
	\address{ Departamento de Matem\'atica, Universidade Federal de Santa Catarina, 88040-900\\
		Florian\'opolis, Brasil}
	\email{d.bagio@ufsc.br}
	
	\author[L. Mart\'inez]{Luis Mart\'{i}nez}
	\address{Escuela de Matematicas, Universidad Industrial de Santander, Cra. 27 Calle 9, ´ UIS
		Edificio 45, Bucaramanga, Colombia} \email{luchomartinez9816@hotmail.com}
	
	\author[H. Pinedo]{H\'ector Pinedo}
	\address{Escuela de Matematicas, Universidad Industrial de Santander, Cra. 27 Calle 9, ´ UIS
		Edificio 45, Bucaramanga, Colombia}
	\email{hpinedot@uis.edu.co}


	\date{\today}
	\begin{abstract}  Let $G$ be a group,  $A=\mathlarger{\mathlarger{\oplus}}_{g\in G}\,A_g$  be an epsilon-strongly graded ring, $R:=A_1$ the homogeneous component associates to the identity of $G$ and   $\mathtt{PicS}(R)$ be  the Picard semigroup of $R$. In the first part of this paper, we prove that the isomorphism class $[A_g]$ is an element of $\mathtt{PicS}(R)$, for all $g\in G$. Moreover, the association $g\mapsto [A_g]$ determines a partial representation of $G$ on $\mathtt{PicS}(R)$ which induces a partial action $\gamma$ of $G$ on the center $Z(R)$ of $R$. Sufficient conditions for $A$ to be an Azumaya $R^{\gamma}$-algebra are presented in the case that $R$ is commutative. In the second part, we study when $B$ is a partial crossed product in the following cases: $B=\operatorname{M}_n(A)$ is the ring of matrices  with entries in $A$, or $B=\operatorname{END}_{A}(M)=\bigoplus_{l \in G} \operatorname{Mor}_A(M,M)_l$ is the direct sum of graded endomorphisms of left graded $A$-module $M$ with degree $l$, or $B=\operatorname{END}_{A}(M)$ where $M=A\otimes_{R}N$ is the induced module of a left $R$-module $N$. Assuming that $R$ is semiperfect, we prove that there exists a subring of $A$ which is an epsilon-strongly graded ring over a subgroup of $G$ and it is graded equivalent to a partial crossed product. 
\end{abstract}

 \maketitle

\setcounter{tocdepth}{1}

\section{Introduction}

Let $G$ be a group  and $A,A'$ be $G$-graded rings. We recall from \cite{H0} that $A$ and $A'$ are  $G$-graded equivalent (or only graded equivalent) provided  that there exists a graded $A$-module $P$ such that $P$ is an $A$-progenerator and $\operatorname{END}_A(P)\simeq A'$ as graded rings, where $\gm{M}$ is given in \eqref{graded-end}. For instance, every $G$-graded ring $A$ is graded equivalent to the matrix ring $\operatorname{M}_n(A)$ with entries in $A$, for some $n\in \N$; see Example 4 of \cite{H}. Among the  important results around graded equivalence  one must highlight the Cohen and Montgomery duality
work of \cite{CM} and  the graded Artin representation work of \cite{GG}. In both, the authors use the following fact: if $A$ and $A'$ are graded equivalent
rings then not only the categories of left graded modules  $\operatorname{A-gr}$ and $\operatorname{A}'\operatorname{\!-gr}$ are equivalent but also  the categories of left modules $\lmod{A}$ and $\lmod{A'}$ are  equivalent.

On the other hand, it is shown in Remark 1.4 of \cite{mnas} that there are Morita equivalent graded rings without being  graded equivalent.  The existence of a category equivalence  between $\operatorname{A-gr}$ and $\operatorname{A}'\operatorname{\!-gr}$ does not imply  that  $A$ and $A'$ are Morita equivalent; see \cite{GG}. We also remark that a related concept to graded equivalence is graded Morita equivalence which was introduced in \cite{H2},  and in  general, graded Morita equivalence is a stronger condition
than graded equivalence; see Corollary 2.8 in \cite{H0}.

Recently was introduced and studied in \cite{NYOP} the notion of epsilon-strongly graded rings which is a natural generalization of strongly graded rings and partial crossed products; see \cite[pg. 2]{NYOP} for more details.  Relevant families of algebras/rings  can be endowed with a grading such that  they become epsilon-strong, are for instance: corner skew Laurent polynomial rings (see Theorem 8.1 of \cite{La}), Leavitt path algebras of finite graphs (see Theorem 1.2 of \cite{NYO}) and  Morita rings (see Section 8 of \cite{NYOP}).

In this work we are interested in studying questions related to epsilon-strongly graded rings. Precisely, we will explore the following problems.

\begin{enumerate}[(Q1)]

\item 	Let $G$ be a group, $A=\mathlarger{\mathlarger{\oplus}}_{g\in G}\,A_g$ be  an epsilon-strongly graded ring and $\pic{R}$ be  the Picard semigroup of $R:=A_1$. Then $A_g$ is an $(R,R)$-bimodule, for all $g\in G$. Is the isomorphism class $[A_g]$ an element of $\pic{R}$? 

\item Let $G$ be a group and be $A=\mathlarger{\mathlarger{\oplus}}_{g\in G}\,A_g$ an epsilon-strongly graded ring. Then the ring of matrices $\operatorname{M}_n(A)$ with entries in $A$ is also an epsilon-strongly graded ring. Characterize when $\operatorname{M}_n(A)$ is a partial crossed product.

\item Let $G$ be a group, $A$  be $G$-graded ring and $M$ be  a left graded $A$-module. Denote by $\operatorname{Mor}_A(M,M)_l$ the abelian group of graded endomorphisms of $M$ with degree $l$ and let  $$\gm{M}=\mathlarger{\mathlarger{\oplus}}_{l \in G} \operatorname{Mor}_A(M,M)_l,$$ which is a graded ring. Determine under what conditions $\gm{M}$ is an epsilon-strongly graded ring/a partial crossed product. We will also address these questions for the particular case where $M=A\otimes_R N$ is the induced module of a left $R$-module $N$.

\item Let $G$ be a group, $A=\mathlarger{\mathlarger{\oplus}}_{g\in G}\,A_g$  be an epsilon-strongly $G$-graded ring such that $R:=A_1$ is semiperfect. Prove that there are a subgroup $H$ of $G$ and an epsilon-strongly $H$-graded ring $A_H$ which is a subring of $A$ and such that $A_H$ is graded equivalent to a partial crossed product. 	
\end{enumerate}

In order to solve the above questions, we consider in Section 2 some notions and results related to partial actions and to epsilon-strongly graded rings that will be useful throughout the work.

In Section 3, we will answer affirmatively to (Q1), that is, $[A_g]$ is an element in $\pic{R}$, for all $g\in G$. Moreover,  the association $g\mapsto [A_g]$ is a unital partial representation of $G$ on  $\pic{R}$. Hence, this partial representation induces an action $\gamma$ of $G$ on the center $Z(R)$ of $R$. We also prove that if $R$ is commutative and $R^\gamma\subset R$ is a partial Galois extension then $A$ is an Azumaya  $R^\gamma$-algebra containing $R$ as a maximal commutative subalgebra.

The question (Q2) is treated in Section 4. It is easy to verify that if $A$ is epsilon-strongly graded then so is $\operatorname{M}_n(A)$; see Example \ref{exam}. Using some new characterizations of epsilon-invertible elements/epsilon-crossed products proved in Lemma \ref{equiv}, we present in Theorem \ref{matcro} necessary and sufficient conditions   for $\operatorname{M}_n(A)$ to be a partial crossed product. 

Section 5 is inspired by  \cite{Da} and it is dedicated to the question (Q3).  Denote by $\cC$ the category of left $A$-graded modules and let $N,M\in \cC$. According to \cite[pg. 269]{KT}, we say that $N$ divides $M$ (we write $ N \,|\, M$ in $\cC$) if $N$ is isomorphic to a graded direct summand of $M^{(n)}$, for some $n\in \N$. This notion is generalized to the following one.  We will say that $N$ semi-divides $M$ if there exists a non-zero graded direct summand $N'$ of $N$ such that $ N' \,|\, M$ in $\cC$ and  $\epsilon_{N'}\circ f=f$ and $g\circ\epsilon_{N'}=g$ (where $\epsilon_{N'}$ is given by \eqref{def-epN'} in $\S\,$\ref{sub-51}), for all  $f\in  \operatorname{Mor}_{ \cC}(M,N)$ and $g\in \operatorname{Mor}_{ \cC}(N,M)$. 
In this case, we denote $N\,|_{\sd}\, M$. When $N\,|_{\sd}\, M$ and $M\,|_{\sd}\, N$ we write $M\sim_{\sd} N$. Consider the graded ring $B=\gm{M}$ defined in \eqref{graded-end} of $\S\,$\ref{sub-51}. 
We prove in Theorem \ref{adivides} that $B$ is an epsilon-strongly graded ring  if and only if $M\sim_{\sd} M(l)$, for all $l \in  \operatorname{supp}(B)$.
Furthermore $M$ and $N$ are called  epsilon-similar if there are non-zero graded direct summands $M'$ of  $M$ and $N'$  of $N$, and morphisms  $f\in  \operatorname{Mor}_{ \cC}(M,N)$ and $g\in\operatorname{Mor}_{ \cC}(N,M)$ such that
\begin{align*}
	&f\circ g=\epsilon_{N'},& &g\circ f=\epsilon_{M'},& &\epsilon_{N'}\circ u=u=u\circ\epsilon_{M'},& &v\circ\epsilon_{N'}=v=\epsilon_{M'}\circ v,&	
\end{align*}	
for all  $u\in  \operatorname{Mor}_{ \cC}(M,N)$ and $v\in \operatorname{Mor}_{ \cC}(N,M)$. In Theorem \ref{pcrp}, we prove the following. 
The ring $B$ is a partial crossed product if and only if $\operatorname{supp}(B)$ is closed under inversion and $M$ and $M(l)$ are epsilon-similar, for all $l \in \operatorname{supp}(B)$.
Suppose that $N$ is a left $R$-module and consider $M= A\otimes_R N$ which is an object in $\cC$ because $M_g= A_g\otimes_R N$, $g\in G$, is a grading to $M$. In Proposition \ref{astor} (resp. Proposition \ref{epcros0}) we present sufficient conditions for $\gm{B}$ to be an epsilon-strongly graded ring (resp. a partial crossed product). \smallbreak

Let $G$ be a group and $A$ and $B$ graded rings over $G$. In order to solve question (Q4), 
 Assume that $A=\mathlarger{\mathlarger{\oplus}}_{g\in G}\,A_g$  be an epsilon-strongly $G$-graded ring such that $R:=A_1$ is semiperfect. In Section 6, it will be proved in Theorem \ref{semicase} that there are a subgroup $G_{\mathcal{E}}$ of $G$ and  an epsilon-strongly $G_{\mathcal{E}}$-graded ring $A_{G_{\mathcal{E}}}$ which is a subring of $A$ and such that $A_{G_{\mathcal{E}}}$ is graded equivalent to a partial crossed product.

\subsection*{Conventions}\label{subsec:conv}
Throughout this work, unless we state, all rings will be considered associative and unital.  
The unit group and the center of a ring $R$ will be denoted respectively by  $\U(R)$ and $Z(R)$. We will work with unital modules/bimodules.
Given a left (resp. right) $R$-module $M$ we will write $\operatorname{Ann}(_RM)=\{r\in R\,:\,rM=0\}$ (resp. $\operatorname{Ann}(M_R)=\{r\in R\,:\,Mr=0\}$) which is an ideal of $R$. The action of an element $r\in R$ on an element $m\in M$ will be denote by  $rm$ (instead of $r\cdot m$).
For a group $G$, the neutral element will be denoted by $1_G$ or simply $1$.
The category of left modules over a unital ring $R$ will be denoted by $\lmod{R}$. For subsets $X$ and $Y$ of the ring $R$, we let $XY$ be the set consisting of finite sums  of elements of the form $xy,$ with $x\in X, y\in Y$. We will assume that each unital subring $S$ of $R$ has the same identity of $R$, that is, $1_S=1_R$. As usual, $\N=\{1,2,\ldots\}$ denotes the natural numbers.

\section{Preliminaries}\label{pre}

In this section we introduce the background about twisted partial actions and epsilon-strongly graded rings that will
be used in the work.

\subsection{Twisted partial  actions}\label{sub-wpa}
Let $G$ be a group and $A$ a ring. According \cite{DES}, a  {\it unital twisted  partial action} of $G$ on a ring $A$ is a  triple $$\af=(\{D_g\}_{g\in G}, \{\afg\}_{g\in G}, \{\omgh\}_{(g,h)\in G\times G}),$$
where $D_g$ is an ideal of $A$ generated by a non-necessarily non-zero central idempotent $1_g$ of $A$, $\afg\colon D_{g\m}\to D_g$
is a ring isomorphism and $\omgh\in \U(D_gD_{gh})$, $g,h\in G$, and the following statements are satisfied: for all $g,h,l\in G$,
\begin{enumerate}[$\,\,\,$(T1)]
\item  $D_1=A$ and $\af_1$ is the identity map of $A$,\vm

\item  $\afg(D_{g\m}D_h)=D_gD_{gh}$,\vm

\item $\afg \circ \afh(t)=\om_{g,h}\afgh(t)\om\m_{g,h}$, for any $t\in D_{h\m} D_{(gh)\m}$,\vm

\item $\om_{1,g}=\om_{g,1}=1_g$,\vm

\item $\afg(\om_{h,l}1_{g\m})\om_{g,hl}=\omgh\om_{gh,l}$.\vm
\end{enumerate}
It follows from (T2) that
\begin{equation}\label{afgh} \afg(1_{g\m}1_h)=1_g1_{gh},\quad g, h\in G.
\end{equation}

The family $\om=\{\omgh\}_{(g,h)\in G\times G}$ is called a {\it twisting} of $\af$ and the above twisted partial action will be denoted by $(\af, \om)$. We say that $\alpha$ is a {\it twisted  global action} provided that $D_g=A$, for all $g\in G$. 

 We also recall from \cite{DES} that the {\it partial crossed product} $A\star_{\af,\om}G$ associated to a unital twisted partial action  $(\af,\om)$ of $G$ on a ring $A$ is the direct sum $\bigoplus\limits_{g\in G}D_g\delta_g$,
in which $\delta_g'$s are symbols, with the usual sum  and multiplication induced by the rule
\begin{equation}\label{parprod}(a_g\delta_g)  (b_h\delta_h) = a_g\af_g(b_h1_{g\m})\om_{g,h}\delta_{gh},\end{equation}for all $g,h\in G$, $a_g\in D_g$ and $b_h\in D_h$. 
Moreover, by Proposition 2.3 (ii) of \cite{DE} and Theorem 2.4 of \cite{DES} that $A\star_{\af,\om}G$ is an associative ring. \smallbreak

\subsection{Epsilon-strongly graded rings}
We start by recalling that  a ring  $A$ is {\it graded} by $G$ (or $G$-graded) if there exists  a family $\{A_g\}_{g\in G}$ of additive subgroups of $A$ such that $A=\mathlarger{\mathlarger{\oplus}}_{g\in G}\,A_g$ and $A_gA_h\subseteq A_{gh}$, for all $g,h\in G$. If $A_gA_h=A_{gh}$, for all $g,h\in G$, then $S$ is said {\it strongly graded}.
The following notion is a generalization of strongly graded ring and was introduced in \cite{NYOP}.

\begin{defi}\label{definitionepsilon}
	Let $A=\mathlarger{\mathlarger{\oplus}}_{g\in G}\,A_g$ be a unital $G$-graded ring.
	We say that $A$ is {\it epsilon-strongly graded by} $G$
	if for each $g \in G$, $A_g A_{g^{-1}}$ is a unital ideal of $A_1$ and 
	$A_g$ is a $(A_gA_{g\m}, A_{g\m}A_g)$-bimodule.
\end{defi}

Its is clear that every unital  strongly graded ring is epsilon-strongly graded. The converse is not true, as we can see in the next.

\begin{exa}\label{ex-tpcp}
	Let $(\alpha,\omega)=(\{D_g\}_{g\in G}, \{\afg\}_{g\in G}, \{\om_{g,l}\}_{(g,l)\in G\times G})$	be a twisted partial action of  $G$ on  $A$. The partial crossed product   $B=A\star_{\af,\om}G$ is an epsilon-strongly graded ring with homogeneous component $B_g=D_g\delta_g$. 
	It is easy to check that $B$ is strongly graded if and only if $\alpha$ is a twisted global action of $G$ on $A$.
\end{exa}

Let $A=\mathlarger{\mathlarger{\oplus}}_{g\in G}\,A_g$ be  a $G$-graded ring. According \cite[Definition 4.5]{CEP}, the ring $A$ is {\it symmetrically graded}  (or {\it partially-strongly graded}) if 
\begin{equation}\label{symg}A_gA_{g\m}A_g=A_g,\quad \text{for all } g\in G.\end{equation} 
Now we present some characterizations of epsilon-strongly graded rings. A proof of this result can be seen in Proposition 3.1 of \cite{NYO} and Proposition 7 of \cite{NYOP}.
\begin{prop}\label{epsilon1}Let $A=\mathlarger{\mathlarger{\oplus}}_{g\in G}\,A_g$ be a  $G$-graded ring. The following statements are equivalent:
	\begin{enumerate}[\rm (i)]
		
		\item $A$ is an epsilon-strongly graded ring,
		\item  $A$ is symmetrically graded  and  
		$A_g A_{g^{-1}}$ is a unital ideal of $A_1$, for all $g\in G$,
		
		\item for every $g \in G$ there exists $\epsilon_g \in A_g A_{g^{-1}}$
		such that $\epsilon_g a = a = a \epsilon_{g^{-1}}$, for all $a\in A_g$,
		
		\item $A_g A_{g^{-1}}$ is a unital ideal of $A_1$
		such that $A_g A_h = A_g A_{g^{-1}} A_{gh} = A_{gh} A_{h^{-1}} A_h$, for all $g,h \in G$.
	\end{enumerate} 
\end{prop}

\begin{rem}\label{rem-epsilon}
	Suppose that $A=\mathlarger{\mathlarger{\oplus}}_{g\in G}\,A_g$ is  an epsilon-strongly graded ring. The element $\epsilon_g$ from Proposition \ref{epsilon1} (iii) is the multiplicative identity element of the ideal  $A_g A_{g^{-1}}$ of $A_1$ and $\epsilon_g\in Z(A_1)$. Particularly, $\epsilon_g$ is unique, for each $g\in G$. We have that $\epsilon_1$ is the multiplicative identity element of both $A_1$ and $A$. Also, $A$ is strongly-graded if and only if $\epsilon_g=1_A,$ for any $g\in G.$ Moreover, by Proposition \ref{epsilon1} (iv),  we have 
	\begin{align}
		\label{eqrep2} A_gA_h=A_gA_{g\m}A_{gh}=\epsilon_gA_1A_{gh}=\epsilon_gA_{gh}
	\end{align}
Similarly, 
\begin{align}\label{eqrep23}
	 A_gA_h=A_{gh}\epsilon_{h\m}
\end{align}	

\end{rem}

In the sequel, we give an example of an epsilon-strongly graded ring that will be useful for our purposes.

\begin{exa}\label{exam}{\bf (Matrix ring)}
	Let $A = \mathlarger{\mathlarger{\oplus}}_{g\in G}\,A_g$ be a $G$-graded ring, $n\in \N$ and $B:=\operatorname{M}_n(A)$ the ring of square matrices of size $n$ with entries in $A$. Observe that $B$ is $G$-graded, where the homogeneous components are given by $B_g=\operatorname{M}_n(A_g)$, for all $g\in G$. Notice that $A$ is epsilon-strongly graded if and only if $B$ is  epsilon-strongly graded. In fact, assume that $A$ is epsilon-strongly graded and let $g\in G$. Then, there exists $\ep_g\in A_gA_{g^{-1}}$ such that $\ep_ga=a=a\ep_{g^{-1}}$, for all $a\in A_g$. Since $\ep_g\in A_gA_{g^{-1}}$, there are $n_g\in \N$, $u_g^{(i)}\in A_g$ and $v_{g^{-1}}^{(i)}\in A_{g\m}$ such that $\ep_g=\sum_{i=1}\limits^{n_g}u_g^{(i)}v_{g^{-1}}^{(i)}$. Denote by $\{e_{ij}\,:\, 1\leq i,j\leq n\}$ the canonical basis of $\operatorname{M}_n(A)$, that is, $e_{ij}$ is the matrix with $1$ in the $(i,j)$-entry and $0$ in the other ones. Then $u_g^{(i)}e_{jj}\in\operatorname{M}_n(A_g)$ and $v_{g^{-1}}^{(i)}e_{jj}\in \operatorname{M}_n(A_{g\m})$. Hence 
	\begin{align*}
		\ep_ge_{jj}=\sum_{i=1}\limits^{n_g}u_g^{(i)}v_{g^{-1}}^{(i)}e_{jj}=\sum_{i=1}\limits^{n_g}u_g^{(i)}e_{jj}v_{g^{-1}}^{(i)}e_{jj}\in \operatorname{M}_n(A_g)\operatorname{M}_n(A_{g\m}),
	\end{align*}
and consequently 
\begin{equation}\label{Ep}E_g:=\ep_gI_n=\ep_ge_{11}+\ldots+\ep_ge_{nn}\in\operatorname{M}_n(A_g)\operatorname{M}_n(A_{g\m}),\end{equation} where	$I_n\in \operatorname{M}_n(A)$ denotes the identity matrix.
It is clear that $E_gX=X=XE_{g^{-1}}$, for all $X\in\operatorname{M}_n(A)$. By Proposition \ref{epsilon1}, $B$ is epsilon-strongly graded.	\vm

Conversely, if $B$ is epsilon-strongly graded then there exists $E_g\in \operatorname{M}_n(A_g)\operatorname{M}_n(A_{g\m})$ such that $E_gX=X=XE_{g^{-1}}$, for all $X\in \operatorname{M}_n(A_g)$ and $g\in G$. For each $a\in A_g$, we have $E_g\cdot ae_{jj}=ae_{jj}$ which implies $E_{g}(j,j)a=a$, where  $E_{g}(j,j)\in A$ is the $(j,j)$-entry of $E_g$.
On the other hand, if $b\in A_{g\m}$ then $be_{jj}\cdot E_g=be_{jj}$ and whence $bE_{g}(j,j)=b$. Thus $E_{g}(j,j)c=c=cE_{g}(j,j)$, for all $c\in A_gA_{g\m}$. Hence $E_g(j,j)$ is an identity element of $A_gA_{g^{-1}}$, for all $j=1,\dots,n$, and it follows that $\ep_g:=E_g(1,1)=\cdots=E_g(n,n)$. Thus, $A$ is epsilon-strongly graded by $G$.
\end{exa}

Let $A=\mathlarger{\mathlarger{\oplus}}_{g\in G}\,A_g$ be a  $G$-graded ring  and denote by $\operatorname{supp}(A)=\{g\in G\,:\,A_g\neq 0\}$ the support of $A$.
The following is clear.

\begin{prop}\label{supp} Let $A=\mathlarger{\mathlarger{\oplus}}_{g\in G}\,A_g$   be a symmetrically graded ring. The following assertions are equivalent: 
	\begin{enumerate}[\rm (i)]
		\item  $g\in \operatorname{supp}(A)$,
		\item  $A_gA_{g\m}\neq 0$,
		\item  $g\m\in\operatorname{supp}(A).$
	\end{enumerate}
\end{prop}

We proceed with the next.
\begin{prop}  \label{epsilon2}Let $A=\mathlarger{\mathlarger{\oplus}}_{g\in G}\,A_g$   be a  graded ring. The following statements are equivalent.
	\begin{enumerate}[\rm (i)]
		\item  $A$ is epsilon-strongly graded by $G$,
		\item  $\operatorname{supp}(A)$  is closed under inversion and for every $g \in \operatorname{supp}(A)$ there is a non-zero idempotent $\epsilon_g \in A_g A_{g^{-1}}$
		such that $\epsilon_g a = a = a \epsilon_{g^{-1}}$, for all $a\in A_g$.
	\end{enumerate}
\end{prop}
\begin{proof}
	\noindent (i) $\Rightarrow$ (ii) Assume that $A$ is epsilon-strongly graded by $G$. By Proposition \ref{epsilon1}, $A$ is symmetrically graded and Proposition \ref{supp} implies that  $\operatorname{supp}(A)$  is closed under inversion. Also, for $g\in \operatorname{supp}(A)$, we have that $A_gA_{g\m}\neq 0$ and whence by  Proposition \ref{epsilon1} there exists an idempotent $0\neq \epsilon_g\in A_{g}A_{g\m}$ that satisfies $\epsilon_g a = a = a \epsilon_{g^{-1}}$, for all $a\in A_g$. 
	
	\noindent (ii) $\Rightarrow$ (i) Let $g\in G$.  If $g\notin \operatorname{supp}(A)$ then  $g\m\notin \operatorname{supp}(A)$. In this case, we define $\epsilon_g=\epsilon_{g\m}=0$.  Since $A_g={0}$, it follows that $ \epsilon_g a=a=a\epsilon_{g\m}$, for all $a\in A_g$. If $g\in \operatorname{supp}(A)$, then there exists a non-zero idempotent $\epsilon_{g}\in A_{g}A_{g^{-1}}$ that satisfies $ \epsilon_g a=a=a\epsilon_{g\m}$, for all $a\in A_g$. From Proposition \ref{epsilon1} (iii) follows that  $A$ is epsilon-strongly graded by $G$.
\end{proof}

Suppose that $A=\mathlarger{\mathlarger{\oplus}}_{g\in G}\,A_g$  is a $G$-graded ring. Recall that a left $A$-module $M$  is  {\it  $G$-graded} if $M=\mathlarger{\mathlarger{\oplus}}_{g\in G}\,M_g$, where $M_g$ is an additive subgroup of $M$, and $A_gM_h\subseteq M_{gh}$, for all $g,h\in G$.  Also, a submodule $N$ of a graded module $M$ is a {\it graded submodule} if $N=\mathlarger{\mathlarger{\oplus}}_{g\in G}\,(N \cap M_g).$ A module morphism $\psi:M\to M'$ between two left $G$-graded modules $M=\mathlarger{\mathlarger{\oplus}}_{g\in G}\,M_g$ and $M'=\mathlarger{\mathlarger{\oplus}}_{g\in G}\,M'_g$ is {\it graded} provided that  $\psi(M_g)\subseteq M'_g$, for all  $g\in G$. We denote by $\operatorname{A-gr}$ the category whose objects are the left graded $A$-modules and the morphisms are the graded module morphisms.\smallbreak

Now we consider a lemma that will be useful later.

\begin{lem}\label{isomul} 
	 Let $A=\mathlarger{\mathlarger{\oplus}}_{g\in G}\,A_g$ be an epsilon-strongly $G$-graded ring,  $R=A_1$  and  $M=\mathlarger{\mathlarger{\oplus}}_{g\in G}\,M_g\in \operatorname{(A,A)-gr}$. Then, for all $g\in G$,
	\begin{enumerate}[\rm (i)]
		\item the multiplication maps $\mu_g\colon  A_g\otimes_{R} M_1 \to\epsilon_g M_g$ and  $\tilde{\mu}_g\colon  A_{g\m}\otimes_{R} M_g \to\epsilon_{g\m}M_1$ are $(R,R)$-bimodule isomorphisms,
        \item  The sextuple $(R\epsilon_{g}, R\epsilon_{g\m}, A_g, A_{g\m},m_{g,g\m}, m_{g\m,g} )$ is a strict Morita context, where $m_{g,g\m}\colon A_g\otimes_{R\epsilon_{g\m}} A_{g\m}\to \epsilon_gR$ and $m_{g\m,g}\colon A_{g\m}\otimes_{R\epsilon_g} A_{g}\to \epsilon_{g\m}R$ are the multiplication maps.
	\end{enumerate}
\end{lem}
\begin{proof}
The item (i) is a consequence of Theorem 3.4 (iii) of \cite{MPS} while (ii) follows from Proposition 48 (d) of \cite{NYOP3}.
\end{proof}
\section{The Picard semigroup}\label{pics}

Throughout this section, $A$ will be a ring. Let $M$ be an $A$-bimodule. The left $A$-module $M$ will be denoted by $_AM$. Similarly, $M_A$ denotes the right $A$-module $M$. If $A=\mathlarger{\mathlarger{\oplus}}_{g\in G}\,A_g$ is an epsilon-strongly graded ring and $R=A_1$ then we will show that the association $g\mapsto [A_g]$ is a unital partial representation of $G$ on the Picard semigroup $\pic{R}$. Consequently we have a partial action $\gamma$ of $G$ on the center $Z(R)$ of $R$. When $R$ is commutative, we will determine sufficient conditions for the $R^{\gamma}$-algebra $A$ to be Azumaya.

\subsection{The Picard semigroup}\label{subs-picard-semi}
Consider $A$-bimodules $M$ and $N$. The  sets $\operatorname{Mor}(M_A,N_A)$ of right $A$-morphisms and  $\operatorname{Mor}(_AM,\,_AN)$ of left  $A$-morphisms are $A$-bimodules with structures given by: for all $m\in M$ and $a\in A$,
\begin{align*}
&(af)(m) = af(m),& &(fa)(m) = f(am),& &f\in \operatorname{Mor}(M_A,N_A),&	\\
&(af')(m) = f'(ma),& &(f'a)(m) = f'(m)a,& &f'\in \operatorname{Mor}(_AM,_AN).&	
\end{align*}
We will denote $M^*=\operatorname{Mor}(_AM,\, _AA)$ and $^*M=\operatorname{Mor}(M_A, A_A)$. We consider the following $A$-bimodule morphisms:
\begin{align*}
&l:A\to \operatorname{End}(M_A),& &l(a)(m)=l_a(m)=am,& &a\in A,\,m\in M,&\\	
&r:A\to \operatorname{End}(_AM),& &r(a)(m)=r_a(m)=ma,& &a\in A,\,m\in M.&	
\end{align*}

In the sequel we recall the notion of partially invertible bimodule given in Section 2.2 of \cite{DDR}. We warn the reader that the notion of partially invertible bimodule has been used in a related sense in  Definition 27 of \cite{NYOP3}.

\begin{defi}\label{defpics} An $A$-bimodule $P$ is said {\it partially invertible} if satisfies:
	\begin{enumerate}[\quad\rm (i)]
		\item $P$ is finitely generated and projective as left and right $A$-module and,\vm
		\item $l$ and $r$ are epimorphisms of $A$-bimodules.
	\end{enumerate}
\end{defi}

Following \cite{DDR}, $\pic{A}$ denotes the set of the isomorphism classes $[P]$ of partially invertible $A$-bimodules, that is,
$$\pic{A}=\big\{[P],\:\, P \, \text {is a 
	partially invertible }\,A\text{-bimodule}  \big\}.$$
The next is Proposition 2.4 of \cite{DDR}.

\begin{prop} $\pic{A}$ is a monoid with multiplication defined by 
	\[ [P][Q] = [P \otimes_A Q], \quad\text{for all } \,[P],[Q]\in \pic{A}.\]
\end{prop}

If $A$ is commutative then the semigroup   $\pic{A}$  consists of the isomorphism classes of finitely generated projective central $A$-bimodules of rank less or equal than one; see Section 3 of \cite{DPP}. 
The following is an immediate consequence of Definition \ref{defpics}.

\begin{lem}\label{principal-unital-ideal}
Let $I$ be a unital ideal of $A$. Then $[I]\in \pic{A}$.
\end{lem}
Denote by $I(A)$ the set of ring isomorphisms between ideals of $A$. Note that $I(A)$ is an inverse semigroup with composition given by: for ideals $I,J,K,L$ of $A$ and isomorphisms $\varphi:I\to J$ and $\psi:K\to L$, we consider the isomorphism $\psi\varphi: \varphi\m(J\cap K)\to \psi(J\cap K)$. 

Let $M$ be an $A$-bimodule and $\theta:A1_{\theta\m}\to A1_{\theta}$ in $I(A)$, where $1_{\theta\m}$ and $1_{\theta}$ are central idempotents of $A$. If $m1_\theta=m$, for all $m\in M$, then $\theta$ induces a structure of right $A$-module on $M_{\theta}:=M$  via 
\begin{align}\label{left-action}
	m\cdot a=m\theta(a1_{\theta\m}),
\end{align}
where the concatenation $m\theta(a1_{\theta\m})$ means the right action in $M_A$ of $\theta(a1_{\theta\m})$ on $m$, for all $a\in A$ and $m\in M$. It is clear that $_{\id}M_{\theta}:=M$ is an $A$-bimodule, where the structure of left $A$-module on $_{\id}M_{\theta}$ is the same of $M$. Analogously one defines the $A$-bimodule  $_{\theta}M_{\id}.$

\begin{prop}\label{1pic} The following statements hold.
	\begin{enumerate}[\rm (i)]

\item If $[P]$ in $\pic{A}$ then:\vm 
		\begin{enumerate}[\rm (a)]
			\item  there exists a central idempotent $e_1\in A$ such that $\operatorname{Ann}(P_A)=Ae_1$ and the map $r:A(1-e_1)\to \operatorname{End}(_AP)$ is an isomorphism of $A$-bimodules,
			\item there exists a central idempotent $e_2\in A$   such that $\operatorname{Ann}(_AP)=Ae_2$ and the map $l:A(1-e_2)\to \operatorname{End}(P_A)$ is an isomorphism of $A$-bimodules. 
		\end{enumerate}
		\item Let $P$ be an $A$-bimodule. If there are an $A$-bimodule $Q$, unital  ideals  $I,J$ of $A$ and morphisms of bimodules $\tau:P\otimes_J Q\to I$ and $\mu:Q\otimes_I P\to J$
		such that $(I, J, P, Q ,\tau,\mu)$ is a strict Morita context then $[P]\in \pic{A}$.\vspace{.05cm}
		\item Let $[P], [Q]\in \pic{A}$. Then $_AP\simeq\! _AQ$ if and only if there is $\theta\in I(A)$ such that $\operatorname{dom}(\theta)$ and $\operatorname{im}(\theta)$ are generated by central idempotents of $A$ and $Q\simeq\!\, _{\id}P_\theta$ as $A$-bimodules.\vspace{.05cm}
		\item Let $\theta:A1_{\theta\m}\to A1_{\theta}$ be an element in $I(A)$ and $M:=A1_{\theta}$, where $1_{\theta\m}$ and $1_{\theta}$ are central idempotents of $A$. Then $[_{\id}M_\theta], [_\theta M_{\id}]\in \pic{A}$.\vspace{.05cm}
		
		\item Assume that $A=\mathlarger{\mathlarger{\oplus}}_{g\in G}\,A_g$ is an epsilon-strongly graded ring and $R=A_1$. Then $[A_g]\in \pic{R}$, for all $g\in G$.
	\end{enumerate}
\end{prop}

\begin{proof} (i) Assume that $[P]$ in $\pic{A}$. Since $r:A \to \operatorname{End}(_AP)$ is an epimorphism of $A$-bimodules and $\ker r=\operatorname{Ann}(P_A)$, it follows that
	$0\rightarrow \operatorname{Ann}(P_A)\rightarrow A \to \operatorname{End}(_AP) \to 0$  is an exact sequence of $A$-bimodules. Notice that the previous sequence splits because $\operatorname{End}(_AP)$ is projective. Hence there exists a central idempotent $e_1$ of $A$ such that $\operatorname{Ann}(P_A)=Ae_1$. Therefore the map $r:(1-e_1)A\to \operatorname{End}(_AP)$ is a bimodule isomorphism. The item (b) is proved similarly.\smallbreak


\noindent (ii) Since $(I, J, P, Q ,\tau,\mu)$ is a strict Morita context, it follows that $_IP$, $P_J$, $_JQ $ and $Q_I$ are  finitely generated and projective modules. Using that $I$ and $J$ are unitals, we have that $I$ and $J$ are direct summands of $A$ and consequently $I$ and $J$ are projective $A$-modules. Thus, $P$ and $Q$ are left and right finitely generated and projective $A$-modules. From \cite[pg. 167]{Ja} follows that  the map $l:I\to \operatorname{End}(P_I)$, $a\mapsto l_a$, is a ring isomorphism. Observe that $\operatorname{End}(P_I)=\operatorname{End}(P_A)$. In fact, the inclusion $\operatorname{End}(P_A)\subset \operatorname{End}(P_I)$ is trivial. For the reverse inclusion, we consider $f\in \operatorname{End}(P_I)$, $x\in P$ and $a\in A$. Then, $f(xa)=f((x1_I)a)=f(x(1_Ia))=f(x)1_Ia=f(x)a$, where $1_I$ is the identity element of $I$. Hence $f\in \operatorname{End}(P_A)$ and $l:A\to \operatorname{End}(P_A)$ is an epimorphism of $A$-bimodules. Similarly one shows that $r:A\to \operatorname{End}(_AP)$ is an epimorphism of $A$-bimodules. Therefore $[P]\in \pic{A}$.\smallbreak

\noindent (iii) Assume that $h:P\to Q$ is an isomorphism of left $A$-modules. Then the map $$h^*: \operatorname{End}(_AQ) \to \operatorname{End}(_AP),\qquad h^*(f)=h\m fh,\quad \text{for all } f\in\operatorname{End}(_AQ),$$ is a ring isomorphism. Since $[P],[Q]\in \pic{A}$, it follows from (ii) that there  are central  idempotents $e_{1,P},\, e_{1,Q}$ of $A$ such that   $\operatorname{Ann}(P_A)=Ae_{1,P}$ and $\operatorname{Ann}(Q_A)=Ae_{1,Q}$. Note that $h^*$ induces a map  $\theta:A(1-e_{1,Q})\to A(1-e_{1,P})$ defined by the following way. Consider $t_{Q}\in  A(1-e_{1,Q})$. By (i), we have that $r_{t_{Q}}\in \operatorname{End}(_AQ)$ and consequently $h^*(r_{t_{Q}})\in \operatorname{End}(_AP)$. Again we obtain from (i) that there is a unique  $t_{P}\in  A(1-e_{1,P})$ such that $r_{t_{P}}=h^*(r_{t_{Q}})$. Hence $\theta:A(1-e_{1,Q})\to A(1-e_{1,P})$ given by $\theta(t_{Q})=t_{P}$ is a well-defined map. Moreover, from $r_{t_{P}}=h^*(r_{t_{Q}})$ follows that
\begin{align}\label{hcond}
h(x)t_{Q}=h(xt_{P})=h(x\theta(t_Q)),\text{   for all   }x\in P.
\end{align} 
We shall check that $\theta\in I(A)$. Let $t_{Q},\, t'_{Q}\in A(1-e_{1,Q})$ and assume that $\theta(t_{Q})=t_{P}$, $\theta(t'_{Q})=t'_{P}$ and $\theta(t_{Q}t'_{Q})=t''_{P}$. Using \eqref{hcond} follows that $t_{Q}t'_Q-t''_Q\in \operatorname{Ann}(P_A)=Ae_{1,P}$. Thus $	t_{Q}t'_Q-t''_Q\in  Ae_{1,P}\cap A(1-e_{1,P})=\{0\}$. Hence $\theta$ preserves products. Similarly we can prove that $\theta$ preserves sums. It is straightforward to check that $\theta$ is bijective and $\theta(1-e_{1,Q})=1-e_{1,P}$. Therefore $\theta\in I(A)$. Notice that $x(1-e_{1,P})=x$, for all $x\in P$, because $\operatorname{Ann}(P_A)=Ae_{1,P}$. Hence $\,_{\id}P_\theta$ is a unital $A$-bimodule. Finally, observe that from  \eqref{hcond} we get $h(x\cdot a)=h\big(x\theta(a(1-e_{1,Q}))\big)=h(x)a(1-e_{1,Q})=h(x)a$, for all $x\in P$ and $a\in A$. Thus $h:\,_{\id}P_\theta\to Q$ is an $A$-bimodule isomorphism. The converse is immediate.\smallbreak

\noindent (iv) Since $M$ is a right $A$-module finitely generated and projective, there is a dual basis $\{m_i, f_i\}_{1\leq i\leq n}$ for  $M_A$. For each $1\leq i\leq n$, consider 
$\varphi_i : M_{\theta}\to A$ given by $\varphi_i(m)= \theta\m(f_i(m)1_{\theta})$. Notice that for $a\in A$ and $m\in M_{\theta}$, we have
$$\varphi_i(m\cdot a)=\theta\m \big(f_i(m\theta( a1_{\theta\m}))1_\theta\big)=\theta\m (f_i(m)1_\theta)a=\varphi_i(m)a,$$
which implies that $\varphi_i$ is a morphism of right $A$-modules. Moreover,
\begin{align*}
	m=\sum\limits_{i=1}^nm_if_i(m)=\sum\limits_{i=1}^nm_if_i(m)1_\theta=\sum\limits_{i=1}^nm_i\theta(\theta\m(f_i(m)1_\theta))=\sum\limits_{i=1}^nm_i\cdot \varphi_i (m).\end{align*}
Hence $\{m_i, \varphi_i\}_{1\leq i\leq n}$ is a dual basis  for  $M_\theta$ and so $M_{\theta}$ is a right $A$-module finitely generated and projective. In order to check that $l:A\to \operatorname{End}(M_{\theta})$ is an epimorphism, take $f\in  \operatorname{End}(M_\theta)$. Then 
$f(m)\theta(a1_{\theta\m})=f(m)\cdot a=f(m\cdot a)=f(m\theta(a1_{\theta\m}))$, for all $a\in A$ and $m\in M$. As $\theta$ is bijective we conclude that $f(mb)=f(m)b$, for all $m,b\in A1_{\theta}$. Thus $f(ma)=f(ma1_\theta)=f(m)a1_\theta=f(m)a$, for all $a\in A$ and $m\in M$. Consequently $f\in \operatorname{End}(M_A)$ and we have $\operatorname{End}(M_\theta)\subset \operatorname{End}(M_A)$. The other inclusion is trivial and we obtain $\operatorname{End}(M_\theta)=\operatorname{End}(M_A)$. From Lemma \ref{principal-unital-ideal} follows that $l:A\to \operatorname{End}(M_A)$ is an epimorphism which implies that $l:A\to \operatorname{End}(M_{\theta})$ is an epimorphism. Using similar arguments, we get that $r:A\to \operatorname{End}(_{\theta}M)$ is surjective. Therefore $[_{\id}M_\theta]\in \pic{A}$. Analogously, $[_{\theta}M_{\id}]\in \pic{A}$.
\smallbreak

\noindent (v) It follows from (ii) above and  Lemma \ref{isomul} (ii). 
\end{proof}

\subsection{Epsilon-strongly graded rings that are Azumaya algebras}\label{ppr} 

We start by  recalling from \cite{DEP} that a  {\it (unital) partial representation of $G$ into an algebra (or a monoid) $S$} is a map $\Phi: G \to S$ which satisfies: 
\begin{enumerate}[\rm (i)]
\item $\Phi (1_G )  = 1_S$,
\item  $\Phi (g\m) \Phi (g) \Phi (h) = \Phi (g\m) \Phi (g h)$,
\item $\Phi (g ) \Phi (h) \Phi (h\m) = \Phi (g h) \Phi ( h \m )$,
\end{enumerate}
for all $g,h\in G$.
 It follows  from (i) and (ii) above  that $\Phi (g) \Phi (g\m) \Phi (g) = \Phi (g),$ for any $g\in G.$

From now on in this subsection, $A=\mathlarger{\mathlarger{\oplus}}_{g\in G}\,A_g$ denotes an epsilon-strongly $G$-graded ring and  $R:=A_1$. Our aim is to give sufficient conditions for $A$ to be an Azumaya algebra. 
The next is an immediate consequence of Proposition \ref{epsilon1} (iii),  Lemma \ref{isomul} (ii)  and Proposition \ref{1pic} (v).

\begin{prop}\label{pr-picard}
The map $\Phi: G \to\pic{R}$ defined by $$\Phi(g)=[A_g],\quad \text{for all } g\in G,$$ is a unital partial representation of $G$ on $\pic{R}$ such that $\Phi (g)\Phi (g\m)=[R\epsilon_g]$. 
\end{prop}

Assume that $R$ is commutative. Following the notation of $\S\,$3.3 in \cite{DDR}, we consider the family of $R$-bimodules isomorphism $f^{\Phi}=\{f^{\Phi}_{g,h}:A_g\otimes_R A_h\to \epsilon_gA_{gh}\,:\,g,h\in G\}$, where  $f^{\Phi}_{g,h}(a_g\otimes a_h)=\epsilon_ga_{g}a_h$, for all $a_g\in A_g$ and $a_h\in A_h$. By \eqref{eqrep2}, $f^{\Phi}_{g,h}$ is a $R$-bimodule isomorphism. Moreover, it is clear that $f^{\Phi}$ satisfies the commutative diagram (29) of \cite{DDR}. Consequently $f^{\Phi}$ is a {\it factor set for $\Phi$} (see \cite{DDR} p. 218).  Thus, we can consider the {\it partial generalized crossed product} $\Delta(\Phi)$ as in \cite[$\S\,$3.3]{DDR}. Explicitly, $\Delta(\Phi)=\oplus_{g\in G} A_g$ is a ring with multiplication defined by
\[a_g\ast_{\Phi} a_h=f^{\Phi}_{g,h}(a_g\otimes a_h)=\epsilon_g a_ga_h\]
Since $\epsilon_ga_g=a_g$, for all $a_g\in A_g$,  we have the following result.

\begin{lem} \label{generalized-pcp}
	If $R$ is commutative then $\Delta(\Phi)\simeq A$ as algebras.	
\end{lem}

Let $g\in G$.  The relation $\epsilon_g \in A_g A_{g^{-1}}$ 
implies that there are a positive integer $n_g$, a subset $\{u_{i,g}\}_{i=1}^{n_g} \subset A_g$ and a subset $\{v_{i,g^{-1}}\}_{i=1}^{n_g}\subset A_{g^{-1}}$
such that $\sum_{i=1}^{n_g} u_{i,g} v_{i,g\m} = \epsilon_g$.
Unless otherwise stated, the elements 
$u_{i,g},\,v_{i,g^{-1}}$ are fixed. We also assume that $n_1 = 1$ and $u_{1,1} = v_{1,1} = 1$.
Then, for each $g\in G$, we obtain the following additive map
\begin{align}\label{defgam}
	&\Gamma_g:A\to A,&  &\Gamma_{g}(a)=\sum_{i=1}^{n_g} u_g^{(i)} av_{g^{-1}}^{(i)}.&
\end{align}
By Proposition \ref{pr-picard} and  Proposition 3.10 of \cite{DDR}, we have a partial action of $G$ on $Z(R)$
\begin{equation}\label{gamma}
	\gamma=(Z(R)\epsilon_{g},\gamma_{g})_{g\in G},
\end{equation} 
where $\gamma_g=Z(R)\epsilon_{g\m}\to Z(R)\epsilon_{g}$  is the restriction of $\Gamma_g$ to $Z(R)\epsilon_{g\m}$.
Moreover, by Lemma 3.11 (i) of \cite{DDR}, the partial action $\gamma$ satisfies  
\begin{equation}\label{eqga}\gamma_g(r)a_g=a_gr,\,\,\,\text{ for all }\,r\in Z(R)\,\text{       and      }\,a_g\in A_g,\,\, g\in G.
\end{equation}

In order to prove the next result we recall some notions. Given a unital partial action $\rho=(T_g, \rho_g)_{g\in G}$ of $G$ on a ring $T$, the {\it subring of invariant elements of $T$} is the subring $T^{\rho}=\{t \in T\,:\,\rho_{g}(t1_{g^{-1}})=t1_g\}$ of $T$.  We also recall from \cite{DFP} that the ring extension $T^{\rho} \subset T$ is a  \textit{partial Galois extension}  if  there are a positive integer $m$ and elements $x_i,y_i\in T$, $1\leq i\leq m$, such that 
\begin{equation*}\label{G2}
	\sum_{i=1}^mx_i\rho_{g}(y_i1_{g^{-1}})=\delta_{1, g},\,\, \text{for all }\, g \in G.
\end{equation*}
The set $\{x_i,y_i\}_{i=1}^m$ is a called a \textit{partial Galois coordinates system} of the extension $T^{\rho} \subset T$.  Also, an algebra $T$ over a commutative ring $S$ is {\it Azumaya} if $Z(T)=S$ and $T$ is separable over $S$.

Let $R$ be a commutative ring. We denote by $\picm{R}$  the submonoid of $\pic{R}$  consisting of isomorphism classes of central $R^\gamma$-bimodules, that is, a class $[P]\in \pic{R}$ belongs to $\picm{R}$ if and only if $\{r\in R^{\gamma}\,:\,rx=xr$, for all $x\in P\}=R^{\gamma}$.  \smallbreak

Now we will prove the main result of this section.

\begin{thm} \label{azumaya}
Assume that $R$ is commutative and let $\gamma$ be the partial action of $G$ on $R$ given in \eqref{gamma}.  If  $R^\gamma\subset R$ is a partial Galois extension then $A$ is an Azumaya  $R^\gamma$-algebra containing $R$ as a maximal commutative subalgebra.
\end{thm}

\begin{proof}
Let $g\in G$. Using the notation introduced in \eqref{left-action}, we have by \eqref{eqga} that the set  $_{\gamma_{g\m}}(R\epsilon_{g\m})_{\id}$ is a  central $R^\gamma$-bimodule via
\begin{align*}
	&r\cdot x=\gamma_{g\m}(r\epsilon_{g})x,& &x\cdot r=xr,& &\text{for all }r\in R\text{ and }x\in R\epsilon_{g\m}.&
\end{align*}
It follows from Proposition \ref{1pic} (iv) and Proposition 6.2 of \cite{DPP}  that 
\[\Phi_0 : G\to \picm{R},\qquad\Phi_0 (g)=[_{\gamma_{g\m}}(R\epsilon_{g\m})_{\id}]\in  \picm{R}\] is a partial representation of $G$. Moreover, $\Phi_0$  satisfies $\Phi_0 (g)\Phi_0 (g\m)=[R\epsilon_g]$ and $\Phi_0 (g)[R\epsilon_{g\m}]=\Phi_0 (g)=[R\epsilon_{g}]\Phi_0 (g)$, for all $g\in G$.  By Theorem  3.8 and  Proposition 6.2 of \cite{DPP},  $\Phi_0$ induces a partial  action $\alpha^*=(X_g, \alpha^*_g)_{g\in G}$ of $G$ on the monoid $\picm{R},$ where $X_g=[R\epsilon_g]\pic{R}\text{   and     }$
$$\alpha^*_g([P])=[_{\gamma_{g\m}}(R\epsilon_{g\m})_{\id}][P][_{\gamma_g}(R\epsilon_{g})_{\id}]=\Phi_0 (g)[P]\Phi_0 (g\m),\,\,\,\, [P]\in X_{g\m}.$$ 
Hence the corresponding group of  $1$-cocycles $Z^1=Z^1(G, \alpha^{\star}, \picm{R})$ is 
\begin{align*}
	Z^1=\{f\in C^1\,:\, f(gh)1^{*}_g=f(g)\alpha^{*}_g(f(h)1^{*}_{g^{-1}}), \, \text{for all } g,h \in G\}
\end{align*}
where $C^1=C^1(G,\alpha^{*}, \pic{R})$ is the set of all maps $f:G\to \pic{R}$ such that $f(g)\in \U(X_g)$,  for all $g\in G$. Notice that $1^{*}_g=[R\epsilon_{g}]$, for all $g\in G$.  We shall construct an element in $Z^1(G,\alpha^{*},\picm{R})$. For each $g\in G$ set $M_g:=A_g{\otimes_R}\,_{ \gamma_{g}}\!(R\epsilon_{g})_{\id}$. Then  items (iv) and (v) of Proposition \ref{1pic} imply that   $[M_g]\in  \picm{R}$. Moreover 
$$[M_g][_{ \gamma_{g\m}}\!(R\epsilon_{g\m})_{\id}]=[M_g]\Phi_0 (g\m)=[M_g]\Phi_0 (g\m)[R\epsilon_g]$$
and   $ M_g\otimes_R R\epsilon_g\simeq M_g$ as $R$-bimodules, where $R\epsilon_g$ is an a sub-bimodule of the bimodule $R$. Consider the map $f:G\to \picm{R}$ given by: $f(g)=[M_g]$, for all $g\in G$. Observe that  $f(g)=\Phi (g)\Phi_0 (g\m)$ in $\picm{R}$ and 
\begin{align*} f(g)\alpha^*_g \big( f(h)[R\epsilon_{g\m}] \big)&=\Phi (g)[\Phi_0 (g\m)\Phi_0 (g)]\Phi (h)\Phi_0 (h\m))([R\epsilon_{g\m}]\Phi_0 (g\m))
	\\&=\Phi (g)[R\epsilon_{g\m}]\Phi (h)\Phi_0 (h\m)\Phi_0 (g\m)
	\\&=\Phi (g)[R\epsilon_{g\m}]\Phi (h)[\Phi_0 (h\m)\Phi_0 (g\m)\Phi_0 (g)]\Phi_0 (g\m)
	\\&=\Phi (g)[R\epsilon_{g\m}]\Phi (h)[\Phi_0 (h\m g\m)\Phi_0 (g)\Phi_0 (g\m)
	\\&=\Phi (g)\Phi (h)\Phi_0 (h\m g\m)[R\epsilon_g]
	\\&=\Phi (gh)[R\epsilon_{h\m}]\Phi_0 (h\m g\m)[R\epsilon_g]\end{align*}\begin{align*}
	\\&=\Phi (gh)\Phi_0 (h\m)\Phi_0 (h)\Phi_0 (h\m g\m)[R\epsilon_g]
	\\&=\Phi (gh)\Phi_0 (h\m)\Phi_0 ( g\m)[R\epsilon_g]
	\\&=\Phi (gh)\Phi_0 (h\m g\m)[R\epsilon_g]
	\\&=f(gh)[R\epsilon_g],
\end{align*}
and $f(g)\alpha^*_g \big( f(h)[R\epsilon_{g\m}] \big)=f(gh)[R\epsilon_g],$ for all $g,h\in G$. Also, by setting $h=g\m$, we get $f(g)\alpha^*_g \big( f(g\m)[R\epsilon_{g\m}] \big)=[R\epsilon_g]$ and
 $f\in Z^1(G, \alpha^{*}, \picm{R})$. Moreover, $f(g)\Phi_0(g)=[A_g][R\epsilon_{g\m}]=[A_g]=\Phi_0(g)$. Therefore,   since $R^\gamma\subset R$ is a partial Galois extension,  it follows from Proposition 6.3 of \cite{DPPR} and Lemma \ref{generalized-pcp} that  $A$ is an Azumaya  $R^\gamma$-algebra containing $R$ as a maximal commutative subalgebra.
\end{proof}

\section {Graded matrix rings as partial crossed products}
Throughout this section, $A=\mathlarger{\mathlarger{\oplus}}_{g\in G}\,A_g$ denotes an epsilon-strongly $G$-graded ring and  $R=A_1$. 
We recall from \cite{NYOP} the following. An element $a_g\in A_g$ is  {\it epsilon-invertible } if there exists $b_{g\m}\in A_{g^{-1}}$ such that  $a_gb_{g\m}=\epsilon_g$ and $b_{g\m}a_g=\epsilon_{g^{-1}}$. If for any $g\in G$ there is an epsilon-invertible element in $A_g$ then the ring $A$ is called {\it an  epsilon-crossed product}. We observe that by Theorem 35 of \cite{NYOP}, epsilon-crossed product and partial crossed product are equivalent notions.

The ring of matrices $\operatorname{M}_n(A)$ with entries in $A$ is epsilon-strongly graded ring; see Example \ref{exam}. In this subsection, using the results of $\S\,$\ref{subs-picard-semi}, we characterize when $\operatorname{M}_n(A)$ is an epsilon-crossed product or equivalently when $\operatorname{M}_n(A)$ is a partial crossed product.

Consider the set $\compo{A}:=\{[A_g]\,:\, g\in G\}$. By Proposition \ref{1pic} (v), $\compo{A}$ is a subset of $\pic{R}$.  Also, if $\theta:R1_{\theta^{-1}}\to R1_{\theta}$ is an element in $I(R)$ then it follows from  Proposition \ref{1pic} (iv) that $[_{\id}(R1_{\theta})_\theta]$ is an element in $\pic{R}$. Therefore, we have the map $\omega:I_u(R)\to \pic{R}$ defined by $\omega(\theta)=[_{\id}(R1_{\theta})_\theta],$ where $I_u(R)$ is the subsemigroup of $I(R)$ of isomorphisms between unital ideals of $R.$  

The next result presents some new characterizations of epsilon-invertible elements and and epsilon-crossed products.

\begin{lem}\label{equiv} The  following assertions hold.
	\begin{enumerate}[\rm (i)]
		\item Let  $g\in G$ and $s_g\in A_g$ for which there are $u_{g\m},v_{g\m}\in A_{g\m}$ such that $s_gu_{g\m}=\ep_g$ and $v_{g\m}s_{g}=\ep_{g\m}.$ Then $u_{g\m}=v_{g\m}$ and $s_g$  is epsilon-invertible.
		\item For each $g\in G$, the following statements are equivalent:
		\begin{enumerate} [\rm (a)]
			\item there exists an epsilon-invertible element in  $ A_g$, 
			\item there is a left $R$-module isomorphism between $A_g$ and $R\ep_g$,
			\item there is a map $\theta: R\epsilon_{g\m}\to R\epsilon_g$  in $I_{u}(R)$ and $A_g\simeq\, _{\id}(R\ep_g)_\theta,$ as $R$-bimodules.
			
		\end{enumerate}
		
		\item Let $H(A)$ be the set of non-zero homogeneous elements of $A$. Then the following assertions are equivalent:
		\begin{enumerate} [\rm (a)]
			\item $A$ is an epsilon-crossed product,
			\item there is a map $\kappa: G\to H(A)$ such that $\kappa(g)\in A_g$ and  $R\kappa(g)=A_g=\kappa(g)R$, for all $g\in G$,
		    \item  $\compo{A}\subseteq {\rm im}\,\omega$,
			\item there is a map $\nu: G\to I_{u}(R)$ such that $A_g\simeq\, _{\id}(R1_{\nu(g)})_{\nu(g)}$ as $R$-bimodules, for all $g\in G$.
			
		\end{enumerate}
		
	\end{enumerate}
\end{lem}

\begin{proof}
\noindent (i)  Let $g\in G$. Then $u_{g\m}=\ep_{g\m}u_{g\m}=v_{g\m}s_{g}u_{g\m}=v_{g\m}\ep_g=v_{g\m},$ as claimed. \smallbreak

\noindent (ii) Fix an element $g\in G$.  To prove that	(a)$\Rightarrow$(b), consider $s_g\in A_g$  an  epsilon-invertible element  with inverse  $s_{g\m}\in S_{g\m}$. The map from $A_g$ to $ R\epsilon_g$ that associates $x_g\mapsto x_gs_{g\m}\in R\epsilon_g$, $x_g\in A_g$, is a left  $R$-module isomorphism. In fact, its inverse map is given by $x\to xs_g\in  S_g$, $x\in R\epsilon_g$. For (b)$\Rightarrow$(c), we observe that by Proposition \ref{1pic} (v), $[A_g]\in \pic{R}$. Hence $[R\epsilon_g]\in \pic{R}$ and $A_g$ and $R\ep_g$ are isomorphic as left $R$-modules.  By the proof of item (iii) of Proposition \ref{1pic}, $A_g\simeq\, _{\id}(R\ep_g)_\theta$, as $R$-bimodules, where $\theta\in I_u(R)$, $\operatorname{dom}(\theta)=R(1-e_1)$, $\operatorname{im}(\theta)=R(1-e_2)$, $Re_1=\operatorname{Ann}((A_g)_R)$ and $Re_2=\operatorname{Ann}((R\ep_{g})_R)$. Note that $\operatorname{Ann}((A_g)_R)=R(1-\epsilon_{g\m})$ and $\operatorname{Ann}((R\ep_{g})_R)=R(1-\epsilon_{g})$. Therefore, $\theta: R\ep_{g\m}\to R\ep_{g}$ and (b)$\Rightarrow$(c) follows. In order to prove (c)$\Rightarrow$(a), consider an   $R$-bimodule isomorphism $\mu_g\colon \, _{\id}(R\ep_g)_\theta\to A_g$, where $\theta: R\epsilon_{g\m}\to R\ep_g$ is an element in $I_{u}(R)$. Consider $s_g:=\mu_g(\ep_g)$. Then $A_g=\mu_g(R\epsilon_g)=Rs_g$. On the other hand, $A_g=\mu_g(R\epsilon_g)=\mu_g(\epsilon_gR)=s_g\theta(R\ep_{g\m})=s_gR\ep_{g}$, 
and consequently
$$R\ep_g=A_gA_{g\m}=s_g(R\ep_{g\m}A_{g\m})\subset s_gA_{g\m}.$$ Hence, there exists an element $s_{g\m}\in A_{g\m}$ such that $s_gs_{g\m}=\ep_g$. Similarly, we have that
$R\ep_{g\m}=A_{g\m}A_{g}\subset A_{g\m}s_g,$ and there exists $u_{g\m}\in A_{g\m}$ such that $u_{g\m}s_g=\ep_{g\m}$. It follows from (i) that $s_g$ is epsilon-invertible and $S$ is an epsilon-crossed product. \smallbreak

\noindent (iii) For	(a)$\Rightarrow$(b) we define $\kappa: G\to H(A)$ by $\kappa(g)=s_g$, for all $g\in G$, where $s_g\in A_g$  is an  epsilon-invertible element  with inverse  $s_{g\m}\in A_{g\m}$.  Then, for each $x\in A_g$, we have  $x=x\ep_{g\m}=(x s_{g\m})s_g\in Rs_g$. Analogously, $A_g=s_gR$. In order to prove (b)$\Rightarrow$(a), consider $s_g:=\kappa(g)\in A_g$, for all $g\in G$. Then $$\epsilon_g\in A_{g}A_{g\m}=s_g(Rs_{g\m})R=s_g(s_{g\m}R)R=s_gs_{g\m}R.$$ Thus, there exists $r\in R$ such that  $\epsilon_g=s_g(s_{g\m}r)$. Analogously, there is $r'\in R$ such that $\epsilon_{g\m}=(s_{g\m}r')s_g$. Hence $s_g$ is epsilon-invertible thanks to (i) above.
For (c)$\Rightarrow$(d), since  $\compo{A}\subseteq {\rm im}\,\omega$, we can choose an element $\nu_g\in I_{u}(R)$ such that $[A_g]=\omega(\nu_g)$, for each $g\in G$. Thus, we have a well-defined map $\nu$ from $G$ to $I_{u}(R)$ given by $\nu(g):=\nu_g$ and $A_g\simeq \, _{\id}(R1_{\nu(g)})_{\nu(g)}$ as $(R,R)$-bimodules. It is clear that (d)$\Rightarrow$(c).
Observe that (a)$\Rightarrow$(d) follows directly from   (a)$\Rightarrow$(c) from (ii) above. Finally, we will show that (d)$\Rightarrow$(b). Let $g\in G$ and $\mu_g\colon \, _{\id}(R1_{\nu(g)})_{\nu(g)}\to A_g$ an $R$-bimodule isomorphism. Define   $\kappa: G\to H(A)$ by $\kappa(g)=\mu_g(1_{\nu(g)})$. Then $A_g=\mu_g(R1_{\nu(g)})=R\kappa(g)$. Notice that 
$1_{\nu(g)}\cdot r=r1_{\nu(g)}$, where $\cdot$ denotes the right action of $R$ on $R1_{\nu_g}$  given in \eqref{left-action} and $r\in R$. Therefore\[\kappa(g)R=\mu_g(1_{\nu(g)})R=\mu_g(1_{\nu(g)}\cdot R )=\mu_g(R1_{\nu(g)})=A_g,\] as desired.
\end{proof}

Now we present the main result of this subsection.

\begin{thm}\label{matcro} Let $n$ be a positive integer. The following assertions are equivalent:
	\begin{enumerate}[\rm (i)]
		\item  $\operatorname{M}_n(A)$ is a partial crossed product,
		\item $\operatorname{M}_n(A)$ is an epsilon-crossed product,
		\item there is  a map  $\nu: G\to I_{u}(\operatorname{M}_n(R))$ such that $\operatorname{M}_n(A_g)\simeq\, _1(\operatorname{M}_n(R)1_{\nu_g})_{\nu(g)},$ as $R$-bimodules,
		
		\item there is  a map  $\kappa: G\to H(\operatorname{M}_n(A))$ such that $\kappa(g)\in \operatorname{M}_n(A_g)$ and  \[\operatorname{M}_n(R)\kappa(g)=\operatorname{M}_n(A_g)=\kappa(g)\operatorname{M}_n(R),\] for all $g\in G$.
		
	\end{enumerate}
	
\end{thm}
\begin{proof} The equivalence (i) $\Leftrightarrow$ (ii) follows from Theorem 35 of \cite{NYOP}. The other equivalences are an immediate consequence of Lemma \ref{equiv} (iii).
\end{proof}

\section{Graded ring endomorphisms}\label{gri}

Throughout this section, $G$ denotes a group, $A=\mathlarger{\mathlarger{\oplus}}_{g\in G}\,A_g$ a 
$G$-graded ring, $R=A_1$ and $\cC:=\operatorname{A-gr}$ the category of left graded $A$-modules.
Given $M\in \cC$, consider the graded ring $\gm{M}=\bigoplus_{l \in G} \operatorname{Mor}_A(M,M)_l$ of graded endomorphisms of $M$ with degree $l$; see details in \eqref{graded-end}. In this section we are interested in characterizing when $\gm{M}$ is epsilon-strongly graded and when it is a partial crossed product. We specialize these problems for induced modules $A\otimes_{R}M\in \cC$ with  $M\in \lmod{R}.$  Notice that according to Proposition 3.15 of \cite{MPS} any simetrically graded $A$-module $N,$ that is $A_{g}A_{g\m}N_g=N_g$ for any $g\in G$ is isomorphic in $\cC$ to an induced module $A\otimes_{R}M,$ for some $M\in \lmod{R}$

\subsection {Graded morphisms}\label{sub-51}
 Let $M,N\in\cC$. Following \cite[pg. 269]{KT}, we say that {\it  $N$ divides $M$} if it is isomorphic to a graded direct summand of $M^{(n)}$, for some $n\in \N$; in this case we  write $ N \,|\, M$. Notice that $ N \,|\, M$ if and only if there are graded morphisms $f_1,\ldots, f_n: M\to N$ and $g_1,\ldots, g_n:N\to M$ such that $\sum_{i=1}^n f_i\circ g_i= \id_N$. If $N|M$ and $M|N$, we write $M\sim N$. 
%

Let $N\in\cC$ and  $N'$ be  a non-zero graded direct summand of $N$. Notice that the projection  $\pi_{N'}:N\to N'$ and the inclusion map $\iota_{N'}:N'\to N$ are graded morphisms and 
\begin{align}\label{def-epN'}
\epsilon_{N'}:=\iota_{N'}\circ\pi_{N'},
\end{align}
is a non-zero idempotent in $\operatorname{End}_{\cC}(N)$.
The next result give us a characterization  of when a graded  direct summand of $N$ divides $M.$

\begin{prop}\label{partdiv} Let $M,N\in\cC$ and  $N'$ a non-zero graded direct summand of $N$. Then $ N' \,|\, M$  if and only if  there are 
	$n\in \N$ and graded morphisms $f_1,\ldots, f_n: M\to N$ and $g_1,\ldots, g_n: N\to M$ such that 
\begin{equation}\label{carsem}\sum_{i=1}^n f_i\circ g_i=\epsilon_{N'}.\end{equation}  
\end{prop}

\begin{proof}
Assume that $N'$ divides $M$. Thus, there are $n\in \N$ and graded morphisms $f'_1,\ldots, f'_n: M\to N'$ and $g'_1,\ldots, g'_n: N'\to M$ such that $\sum_{i=1}^n f'_i\circ g'_i=\id_{N'}$.
Take
$f_i:=\iota_{N'}\circ f'_i$ and $g_i:=g'_i\circ\pi_{N'}$, for all $1\leq i\leq n$. Notice that $f_i: M\to N$, $g_i: N\to M$ are graded morphisms and 
$$\sum_{i=1}^n f_i\circ g_i=\sum_{i=1}^n \iota_{N'}\circ(f'_i \circ g'_i)\circ \pi_{N'}= \epsilon_{N'}.$$
Conversely, setting $g_i':=g_i\circ\iota_{N'}$ and $f_i':= \pi_{N'}\circ f_i$,  then
$$\sum_{i=1}^n f'_i\circ g'_i=\sum_{i=1}^n  \pi_{N'}\circ( f_i\circ g_i)\circ\iota_{N'}= \pi_{N'}\circ\ep_{N'}\circ\iota_{N'}=\id_{N'}.$$ Thus $N'$ divides $M$. 
\end{proof}

Let $M=\mathlarger{\mathlarger{\oplus}}_{g\in G}\,M_g$  and $N=\mathlarger{\mathlarger{\oplus}}_{g\in G}\,N_g$ be  objects in $\cC$. For $l \in G$,  we denote by  $M(l)$ the left {\it $l$-suspension } of $M$, that is, the  $G$-graded module $M$ with $g$-homogeneous component $\big(M(l)\big)_g=M_{gl}$. 
We recall that  $f\in \operatorname{Mor}_A(M,N)$ is a {\it graded morphism of degree $l$} if $f(M_g)\subset N_{gl},$ for all $g\in G$.
The additive subgroup  of $ \operatorname{Mor}_A(M,N)$ consisting of all graded morphism from $M$ to $N$ of degree $l$ is denoted by $ \operatorname{Mor}_A(M,N)_l$. Observe that $\sum_{l \in G}\operatorname{Mor}_A(M,N)_l$ is a direct sum of additive groups and then we put

$$\operatorname{Mor}_A(M,N):=\bigoplus_{l \in G} \operatorname{Mor}_A(M,N)_l.$$ Hence $\operatorname{Mor}_A(M,N)$  is a $G$-graded abelian group and 

\begin{align*} \operatorname{Mor}_A(M,N)
	=\bigoplus_{l \in G} \operatorname{Mor}_{\cC}(M,N(l))
	=\bigoplus_{l \in G}\operatorname{Mor}_{\cC}(M(l\m),N).\end{align*} 
Moreover, the set   
\begin{align}\label{graded-end}
	\gm{M}:=\operatorname{Mor}_A(M,M)
\end{align}
with the usual addition and multiplication given by $uv=v\circ u$, for all $u,v\in \gm{M}$, is a $G$-graded ring with $\gm{M}_1= \operatorname{End}_{ \cC}(M)$.
In order to  give an example  of when $\gm{M}$ is epsilon-strongly graded, we proceed with the next.
\begin{prop}\label{equivi} Let $M,N\in\cC$,  $N'$ a non-zero graded direct summand of $N,$ $\epsilon_{N'}$ defined as in  \eqref{def-epN'} and  $f \in  \operatorname{Mor}_{ \cC}(M,N)$ (resp. $g\in  \operatorname{Mor}_{ \cC}(N,M)$). Then the following assertions are equivalent:
	\begin{enumerate} [\rm (i)]
		\item  $\epsilon_{N'}\circ f=f$ (resp. $g\circ\epsilon_{N'}=g$) ; \vm
		\item  $\operatorname{im}\, f\subseteq N'$ (resp. $\ker \epsilon_{N'}\subseteq \ker g$) ; 
		\item there is  $f'\in  \operatorname{Mor}_{ \cC}(M,N')$  such that $f=\iota_{N'}\circ f'$ (resp. there is  $g'\in \operatorname{Mor}_{ \cC}(N',M)$ such that $g=g'\circ \pi_{N'}$). 
	\end{enumerate} 
%
\end{prop}
\begin{proof} 
Note that  the implications (iii) $\Rightarrow$ (i) $\Rightarrow$ (ii) are   immediate. For (ii) $\Rightarrow$ (iii), since  $\operatorname{im}\, f\subseteq N'$, we can take the corestriction $f':M\to N'$  of $f$ to $N'$. Notice that $f=\iota_{N'}\circ f'$.   Similarly, consider the restriction $g':N'\to M$ of $g$ to $N'$. Since $N=N'\oplus\ker\epsilon_{N'}$ and $\ker \pi_{N'}=\ker \epsilon_{N'}\subset \ker g$, it follows that $g= g'\circ\pi_{N'}$. 
\end{proof}


\begin{exa}\label{first}  Let $\Bbbk$ be a field and consider the vector space $V=\Bbbk^3$ endowed with a  $\Z$-grading via $V_{-1}=\Bbbk\times \{0\}\times \{0\}, V_0=\{0\}\times \Bbbk\times \{0\},  V_{1}=\{0\}\times \{0\}\times \Bbbk$ and $V_n=\{0\},$ for $n\notin \{-1,0,1\}$.  
It is clear that $\operatorname{END}_{\Bbbk}(V)_m=\{0\}$ if $m\notin \{\pm 2,\pm 1,0\}$. Now we define the following $\Bbbk$-linear operators on $V,$ for all  $(x,y,z)\in V$,
\begin{align*}
	&\epsilon_0={\rm id}_V,& &\epsilon_{1}(x,y,z)=(x,y,0),& &\epsilon_{2}(x,y,z)=(x,0,0),&\\
	&\epsilon_{-1}(x,y,z)=(0,y,z),& &\epsilon_{-2}(x,y,z)=(0,0,z),& &u_{-1}(x,y,z)=(y,z,0),& \\
	&u_{-2}(x,y,z)=(z,0,0),& &v_1(x,y,z)=(0,x,y),&  &v_2(x,y,z)=(0,0,x),&
\end{align*}
 then $\epsilon_{-i}= v_{i}\circ  u_{-i}\in \operatorname{END}_{\Bbbk}(V)_{-i}\operatorname{END}_{\Bbbk}(V)_{i}$ and $\epsilon_{i}=u_{-i}\circ v_{i}\in \operatorname{END}_{\Bbbk}(V)_{i}\operatorname{END}_{\Bbbk}(V)_{-i}$, for $i\in\{1,2\}$.   Moreover, for $i=1,2$, we have:
\begin{align*}
	&\ker \epsilon_{i} \subset \ker f_i,& &\Ima f_i \subset \Ima \epsilon_{-i},& &\text{for all }f_i\in \operatorname{END}_{\Bbbk}(V)_i,&\\
	&\ker \epsilon_{-i} \subset \ker f_{-i},& &\Ima f_{-i} \subset \Ima \epsilon_{i},& &\text{for all }f_{-i}\in \operatorname{END}_{\Bbbk}(V)_{-i}.&
\end{align*}
Thus, Proposition \ref{equivi} implies that $\epsilon_i f =f=f\epsilon_{-i}$, for all $f\in \operatorname{END}_{\Bbbk}(V)_i$ and $i\in\{\pm 2,\pm 1,0\}$. Hence, by Proposition \ref{epsilon2},
$\operatorname{END}_{\Bbbk}(V)$ is epsilon-strongly graded.
\end{exa}

We are interested in characterizing when the graded ring $\gm{M}$ defined in \eqref{graded-end} is epsilon-strongly graded. For this, the following notion will be helpful.

\begin{defi} \label{semi-divides} Let $M,N\in\cC$. We will say that $N$ {\it  semi-divides} $M$ if there exists a non-zero graded direct summand $N'$ of $N$ such that $ N' \,|\, M$ in $\cC$ and  $\epsilon_{N'}\circ f=f$ and $g\circ\epsilon_{N'}=g$, for all  $f\in  \operatorname{Mor}_{ \cC}(M,N)$ and $g\in \operatorname{Mor}_{ \cC}(N,M)$. In this case, we denote $N\,|_{\sd}\, M$. When $N\,|_{\sd}\, M$ and $M\,|_{\sd}\, N$ we write $M\sim_{\sd} N$.  
\end{defi}

\begin{thm} \label{adivides}Let $M\in \cC$. The graded ring $B=\gm{M}$ defined in \eqref{graded-end} is epsilon-strongly graded  if and only if $M\sim_{\sd} M(l)$, for all $l \in  \operatorname{supp}(B)$.
\end{thm}
\begin{proof}
By Proposition \ref{epsilon2}, if $B$ is  epsilon-strongly $G$-graded then  there exists  a non-zero idempotent element
$\epsilon_{l}\in B_l B_{l\m}$ 
such that $u=\epsilon_{l}u=u\circ\epsilon_{l}$ and $v=v\epsilon_{l}=\epsilon_{l}\circ v$, for all $l \in  \operatorname{Supp}(B)$, $u\in B_{l}$ and $v\in B_{l\m}$.
Let $u_1,\cdots, u_n\in B_{l}$ and $v_1,\cdots, v_n\in B_{l\m}$ be such that 
\begin{equation}\label{idemp}
\epsilon_l=\sum_{i=1}^n  u_iv_i=\sum_{i=1}^n v_i\circ u_i.\end{equation}
Since $\epsilon^2_{l}=\epsilon_{l}$, we have that $M=M'\oplus M''$, where $M'=\operatorname{im} \epsilon_{l}$ and $M''=\ker\epsilon_{l}$. Observe that $M'$ is an non-zero graded direct summand of $M$ and $\epsilon_{l}=\epsilon_{M'}$,  with $\epsilon_{M'}$ given by \eqref{def-epN'}.
Using that 
\begin{align*}
	&B_{l}= \operatorname{Mor}_{\cC}(M,M(l)),& 	&B_{l\m}= \operatorname{Mor}_{\cC}(M(l),M),&
\end{align*}
we obtain from  \eqref{idemp} and Proposition \ref{partdiv}  that  $M$  quasi-divides $M(l)$. Moreover,
$u=u\circ \epsilon_{l}=u\circ \epsilon_{M'}$ and $v=\epsilon_{l}\circ v=\epsilon_{M'}\circ v$, for all $u\in B_{l}= \operatorname{Mor}_{\cC}(M,M(l))$ and $v\in B_{l\m}= \operatorname{Mor}_{\cC}(M(l),M)$. Hence $M \,|_{\sd}\, M(l)$. Applying a similar argument for $\epsilon_{l\m}$ we conclude that  $M(l)\,|_{\sd}\, M$. Thus  $M\sim_{\sd} M(l)$, for all $l \in  \operatorname{supp}(B)$.\vm

Conversely, consider $l \in  \operatorname{Supp}(B)$. By assumption, $M\,|_{\sd}\, M(l)$ and whence there is a non-zero graded direct summand $M'$ of $M$ such that $M'$ divides $M(l)$, $\epsilon_{M'}\circ u=u$ and $v\circ\epsilon_{M'}=v$, for all $u\in \operatorname{Mor}_{\cC}(M(l),M)$ and $v\in \operatorname{Mor}_{\cC}(M,M(l))$. By Proposition \ref{partdiv}, there are $n\in \N$ and graded morphisms $u_1,\ldots, u_n: M(l)\to M$ and $v_1,\ldots, v_n: M\to M(l)$ such that $\sum_{i=1}^n v_iu_i=\sum_{i=1}^n u_i\circ v_i=\epsilon_{M'}$. Since $v_i\in B_{l}$ and $u_i\in B_{l\m}$ we conclude that $\epsilon_{l}:=\epsilon_{M'}\in B_{l}B_{l\m}$ satisfies $\epsilon_{l}v=v\circ\epsilon_{l}=v$, for all $v\in B_{l}$. Observe that $B_{l\m}\neq 0$ because $0\neq \epsilon_{l}\in B_{l}B_{l\m}$ and consequently $\operatorname{supp}(B)$ is closed under inversion. Similarly, from $M(l)\,|_{\sd}\, M$ we obtain $\ep_{l\m}\in B_{l\m}B_{l}$ that satisfies $u\epsilon_{l\m}=\epsilon_{l\m}\circ u=u$, for all $u\in B_{l\m}$. Hence Proposition \ref{epsilon2} implies that $B$ is epsilon-strongly graded by $G$.
\end{proof}

As a consequence of the previous theorem we recover below (4.5) of \cite{Da}.

\begin{cor} \label{asto} Let $M\in \cC$. The ring $\gm{M}$ defined in \eqref{graded-end} is strongly graded  if and only if $M\sim M(l)$, for all $l \in  G$.
\end{cor}

\begin{proof} If $B=\gm{M}$ is strongly graded then by Proposition \ref{adivides} we have  $M\sim_{\rm sd} M(l)$ for all $l \in  G$. Moreover,  the fact that $B$ is strongly graded implies that the $A$-module  $M'$ constructed in the proof of  Proposition \ref{adivides} coincides with $M$ (in this case, we have that $\epsilon_l=\id_M$, for all $l\in G$). Thus $M\sim M(l)$, for all $l \in  G$.
	
Conversely, assume that $ M\sim M(l)$, for all $l \in  G$. Then $B$ is epsilon-strongly graded with $\epsilon_l=\id_{M}$, for each $ l\in G$,  and whence $B$ is strongly graded.
\end{proof}

Given $M\in\lmod{R}$ the induced module  $A\otimes_{R} M\in \lmod{A}$ is $G$-graded with homogeneous component $(A\otimes_{R} M)_g:=A_g\otimes_{R} M$, for all $g\in G$. This association defines a functor 
\begin{align}
	\label{functor1} &\indf:\lmod{R}\to \cC,& &\indf M=A\otimes_{R} M,& &\text{for all }M\in \lmod{R}.& 
\end{align}
For a morphism $f:M\to N$ in $\lmod{R}$, the map $\mathtt{Ind_R^A}(f)=\id_{A}\otimes f:A\otimes_{R} M\to A\otimes_{R} N$ is a morphism in $\cC$.

\begin{rem}\label{Dades} Dade's well-known result \cite[Theorem 3.8]{Da} establishes  that  $A$ is strongly graded if and only if  the functors $(-)_1: \cC\to \lmod{R}$ given by $M=\oplus_{g\in G}M_g\to M_1 $ and   $\indf$ defined in \eqref{functor1} determine a category equivalence. Thus several problems in category  $\cC$ can be translated to $\lmod{R}$ and vice versa.
\end{rem}

To the next result we need some extra notations. For each $l\in G$, we consider the set $\operatorname{supp}_l(A):=\left\{g\in G\,:\,(g,gl)\in \operatorname{supp}(A)\times \operatorname{supp}(A)\right\}$.
Also, for $M\in \lmod{R}$ and  $l\in \operatorname{supp}(A)$ we denote
\begin{align}\label{modules_N}
&N^{(l)}:=\displaystyle\bigoplus_{g\in \operatorname{supp}_l(A)}A_{gl}\otimes_R M, &&N_{(l)}:=\displaystyle\bigoplus_{g\in \operatorname{supp}_l(A)}A_{g}\otimes_R M.&	
\end{align}
Observe that $N^{(l)}$ and  $N_{(l)}$ are graded submodules of $\indf{M}=A\otimes_{R} M$. 

\begin{prop}\label{astor} Let $M\in \lmod{R}$.  Assume that $N^{(l)}\simeq N_{(l)}\simeq N^{(l\m)}$ (as graded modules), for all $l\in \operatorname{supp}(A)$.  Then  $C=\gm{\indf{M}}$ is epsilon-strongly  graded. 
\end{prop}
\begin{proof}
Let $N:=\indf{M}=A\otimes_{R} M$ and  $l\in \operatorname{supp}(A)$. Observe that the projection map $\pi_{N^{(l)}}:N\to N^{(l)}$ is an element of $\operatorname{Mor}_{\cC}(N,N(l))$. 
Since 
\[\displaystyle N^{(l)}=\bigoplus_{g\in \operatorname{supp}_l(A)}A_{gl}\otimes_R M\simeq\displaystyle\bigoplus_{g\in \operatorname{supp}_{l}(A)}A_{g}\otimes_R M=N_{(l)},\]
it follows that the  inclusion $\iota_{N^{(l)}}:N^{(l)}\to N$, given by $\iota_{N^{(l)}}\left(A_{gl}\otimes_R M\right)=A_{g}\otimes_R M$, belongs to $\operatorname{Mor}_{\cC}(N,N(l\m))$. Hence $\epsilon_{l}:=\epsilon_{N^{(l)}}=
\iota_{N^{(l)}}\circ \pi_{N^{(l)}}= \pi_{N^{(l)}}\iota_{N^{(l)}}\in C_lC_{l\m}$. Notice that $\ker \epsilon_l\subset \ker u$, for all $u\in \operatorname{Mor}_\cC(N,N(l))=C_l$. Then, Proposition \ref{equivi} implies  that $\epsilon_{l}u=u\circ \epsilon_{l}=u$. Also, $\operatorname{im}\, u\subseteq N^{(l)}\simeq N^{(l\m)}$ and we obtain  from Proposition \ref{equivi}   that $u\epsilon_{l\m}=\epsilon_{l\m}\circ u=u$. 
In order to complete the proof, by Proposition \ref{epsilon2}, it is enough to check that $\operatorname{supp}(C)$ is closed under inversion. Let $l\in \operatorname{supp}(C)$. Then, there exists a non-zero element $u\in \operatorname{Mor}_\cC(N,N(l))=C_l$. Thus, for some $g\in \operatorname{supp}(A)$, we have that $u(A_g\otimes_{R} M)=A_{gl}\otimes_R M\neq 0$ and consequently $\iota_{N^{(l)}}$ is a non-zero element of $\operatorname{Mor}_\cC(N,N(l\m))=C_{l\m}$. Hence $l\m \in \operatorname{supp}(C)$.
\end{proof}


\subsection{Graded ring endomorphisms as partial crossed products}

Let $M\in\cC$. In order to characterize when $\gm{M}$ defined in \eqref{graded-end} is a partial crossed product we start with the following.

\begin{defi} \label{def:qd} Let $M,N\in\cC$. We will say that $M$ and $N$ are {\it epsilon-similar} if there are non-zero graded direct summands $M'$ of  $M$ and $N'$  of $N$, and morphisms  $f\in  \operatorname{Mor}_{ \cC}(M,N)$ and $g\in\operatorname{Mor}_{ \cC}(N,M)$ such that
\begin{align}\label{eq-def}
&f\circ g=\epsilon_{N'},& &g\circ f=\epsilon_{M'},& &\epsilon_{N'}\circ u=u=u\circ\epsilon_{M'},& &v\circ\epsilon_{N'}=v=\epsilon_{M'}\circ v,&	
\end{align}	
for all  $u\in  \operatorname{Mor}_{ \cC}(M,N)$ and $v\in \operatorname{Mor}_{ \cC}(N,M)$. 
\end{defi}

\begin{thm} \label{pcrp} Let $M\in\cC$ and $B=\gm{M}$ the $G$-graded ring defined in \eqref{graded-end}. The following statements are equivalent:
	\begin{enumerate}[\rm (i)]
		\item $B$ is a partial crossed product,\vm
		\item $B$ is an epsilon-crossed product,\vm
		\item $\operatorname{supp}(B)$ is closed under inversion and $M$ and $M(l)$ are epsilon-similar, for all $l \in \operatorname{Supp}(B)$.
\end{enumerate}
\end{thm}

\begin{proof}
The equivalence (i) $\Leftrightarrow$ (ii) follows from Theorem 35 of \cite{NYOP}. In order to prove that (ii) $\Rightarrow$ (iii), assume that $B$ is  an epsilon-crossed product and $l \in \operatorname{supp}(B)$. Since $B$ is epsilon-strongly graded, it follow by Proposition \ref{epsilon2} that $\operatorname{supp}(B)$ is closed under inversion. Also,  there is an  epsilon-invertible element $f_l\in B_{l}=\operatorname{Mor}_{\cC}(M,M(l))$. Thus, there exists  $h_{l\m}\in B_{l\m}= \operatorname{Mor}_{\cC}(M(l),M)$ such that $f_lh_{l\m}=h_{l\m}\circ f_l=\epsilon_{l}$ and $ h_{l\m}f_l=f_l\circ h_{l\m}=\epsilon_{l\m}$. Take $M'=\operatorname{im}\epsilon_l$ and $N'=\operatorname{im}\epsilon_{l\m}$ which are non-zero graded direct summands of $M$ and $M(l)$, respectively.  Since  $\epsilon_{M'}=\epsilon_l$ and $\epsilon_{N'}=\epsilon_{l\m}$, we have that $u=\epsilon_lu=u\circ\epsilon_l=u\circ\epsilon_{M'}$ and $v=v\epsilon_l=\epsilon_l\circ v=\epsilon_{M'}\circ v$, for all  $u\in B_{l}=\operatorname{Mor}_{\cC}(M,M(l))$ and $v\in B_{l\m}=\operatorname{Mor}_{\cC}(M(l),M)$. Similarly,  $u=\epsilon_{N'}\circ u$ and $v=v\circ\epsilon_{N'}$. Hence, $M$ and $M(l)$ are epsilon-similar.\vm

To prove (iii) $\Rightarrow$ (ii), consider $l \in \operatorname{supp}(B)$ and assume that $M$ and $M(l)$ are epsilon-similar. Then there are non-zero graded direct summands $M'$ of  $M$ and $M''$  of  $M(l)$ and  morphisms  $f_l\in  \operatorname{Mor}_{ \cC}(M,M(l))$ and $h_{l\m}\in\operatorname{Mor}_{ \cC}(M(l),M)$  that satisfy \eqref{eq-def}. Thus $f_l\in B_l$ is epsilon-invertible. If $l\notin\operatorname{supp}(B)$  then $l\m\notin\operatorname{supp}(B)$ and the null morphism $f_l\in B_l$ is trivially  epsilon-invertible.
\end{proof}

\begin{exa}
Let $B=\operatorname{END}_{\Bbbk}(V)$ be the epsilon-strongly graded considered in Example \ref{first}. It follows from Theorem \ref{pcrp} that $B$ is a partial crossed product.
\end{exa}

We say that $M\in \lmod{R}$  is called a  {\it $G$-invariant module} if $A_g \otimes_R M\simeq M$ in $\lmod{R}$, for all $g\in \operatorname{supp}(A)$.

\begin{prop}\label{epcros0} Let $M\in \lmod{R}$ be a $G$-invariant module. Then $C=\gm{\indf M}$ is a partial crossed product.
\end{prop}

\begin{proof} 
 We claim that $C$ is epsilon-strongly graded. In fact, let $l\in \operatorname{supp}(A),$  $N:=\indf{M}$ and $N^{(l)},\, N_{(l)}$ the graded submodules of $N$ given in \eqref{modules_N}. It is clear that $N^{(l)}\simeq N_{(l)}$ (as graded modules) because $M$ is $G$-invariant. Also, using that $\operatorname{supp}_l(A)l=\operatorname{supp}_{l\m}(A)$ we obtain that $N^{(l)}\simeq N^{(l\m)}$ (as graded modules). Thus, it follows from  Proposition \ref{astor} that $C$ is epsilon-strongly graded with  $\epsilon_l=\iota_{N^{(l)}}\circ \pi_{N^{(l)}}$, and $\pi_{N^{(l)}}$ and $\iota_{N^{(l)}}$ are respectively the projection and the inclusion maps.  To prove that $C$ is a partial crossed product, take $g\in \operatorname{supp}(A)$. By assumption, there exists an isomorphism $\phi_g:A_g\otimes_R M\to  M$. Then, $\phi_{g,h}:A_g\otimes_R M\to A_h\otimes_R M$ defined by $\phi_{g,h}=\phi_{h}^{\m}\circ\phi_g$, is an isomorphism of $R$-modules with inverse $\phi_{g,h}^{-1}=\phi_{h,g}$, for all $g,h\in \operatorname{supp}(A)$.
We also define $\phi_l \in C_l$ by $$ \phi_l\,\, =\sum\limits_{g\in  \operatorname{supp}(A)}\phi'_{g,l} \,\,\,\,\,\,\,\,{\rm where}\,\,\,\,\,\,\,\,
\phi'_{g,l}=
\begin{cases}
	\phi_{g,gl},\text{     if     } g\in \operatorname{supp}_l(A),\\
	0, \text{  \quad   otherwise.     }
\end{cases}
$$
It is clear that $\phi_l$ is epsilon-invertible with $\phi_{l}^{\m}\in C_{l\m}$
given by $ \phi_{l}^{-1} =\sum\limits_{gl\in  \operatorname{supp}(A)}\phi'_{gl,l\m}$.
Indeed,  if $x\,\,\,=\sum\limits_{g\in \operatorname{supp}(A)}a_g\otimes m\in A\otimes_RM,$ then
\begin{align*}
	\phi_{l}^{\m}\circ \phi_{l}(x)&=\sum\limits_{g\in \operatorname{supp}_l(A)}a_g\otimes m=\epsilon_l(x),
\end{align*}
that is, $\phi_{l}\phi_{l\m}=\phi_{l\m}\circ\phi_{l}=\epsilon_l$. Similarly, $\phi_{l\m}\phi_{l}=\epsilon_{l\m}$ and it follows that $C$ is an epsilon-crossed product. Hence, Theorem \ref{pcrp} implies that $C$ is a partial crossed product.
\end{proof}

\section{Epsilon-strongly graded subrings which are graded equivalent to partial crossed products} \label{gradedequ}

Let $A$ and $B$ be graded rings over a group $G$. Following \cite{H0}, we say that  $A$ and $B$ are {\it $G$-graded equivalent} if there is a graded $A$-module $P$ such that $P$ is an $A$-progenerator  and  $\gm{P}\simeq B$ as graded rings, where $\gm{P}$ is given in \eqref{graded-end}.

Throughout this section,  $A=\mathlarger{\mathlarger{\oplus}}_{g\in G}\,A_g$ denotes  an epsilon-strongly graded ring and $R:=A_1$. Assume that  $R$ is semiperfect and let $\mathcal{E}=\{e_1,\ldots,e_n\}$ as in \eqref{E} below. We will see  in Theorem \ref{semicase} that there are a subgroup $G_{\mathcal{E}}$ of $G$  and an epsilon-strongly $G_{\mathcal{E}}$-graded ring $A_{G_{\mathcal{E}}}$ which is a subring of $A$ and  $A_{G_{\mathcal{E}}}$ is graded equivalent to a  partial crossed product.

\begin{rem} Let $P$ be a free $A$-module of rank $n$. Then $P\simeq A^n$, for some $n\in \mathbb{N}$. Hence, $P$ is $G$-graded with the grading induced by $A$. Also, $P$ is a progenerator in $\lmod{A}$ such that $\gm{P}\simeq \operatorname{M}_n(A)$. Then,  $A$ is graded equivalent to the  partial crossed product $\operatorname{M}_n(A)$ if and only if one of the conditions in Theorem \ref{matcro} hold.
\end{rem}


From now on in this section, $R$ is assumed to be semiperfect. By Proposition 27.10 of \cite{AF}, there exists a set  
\begin{equation}\label{E} \mathcal{E}=\{e_1,\ldots,e_n\}
\end{equation}  
of pairwise orthogonal primitive idempotents of $R$ such that $\{Re_1,\ldots,Re_n\}$ is an irredundant set of representatives of the indecomposable projective left $R$-modules.

Observe that if $e\in \mathcal{E}$ then $\epsilon_ge=0$ or $\epsilon_ge=e$. In fact, by Remark \ref{rem-epsilon}, $\ep_g$ is a central idempotent of $R$ and hence $\ep_ge$ is an idempotent of $R$. From $e=\ep_ge+(1-\ep_g)e$ follows that $\ep_ge=0$ or $\ep_ge=e$ since $e$ is primitive. Moreover, we have the following.

\begin{lem}\label{ege} Let $\mathcal{E}$ be the set given in \eqref{E}, $e\in \mathcal{E}$ and $g\in G$. The following statements are equivalent:
\begin{enumerate}[\rm (i)]
\item $\epsilon_ge=e$,\vm
\item $A_{g}e\neq \{0\}$,\vm
\item $A_{g\m}e\neq \{0\}$.
\end{enumerate}
\end{lem}

\begin{proof}
If $\epsilon_ge\neq e$ then $\epsilon_ge=0$. Hence $ae=(a\epsilon_{g})e=0$, for all $a\in A_{g\m}$. Consequently $A_{g\m}e= \{0\}$ and the implication (iii) $\Rightarrow$ (i) follows. In order to prove that (i) $\Rightarrow$ (ii), assume that  $A_{g}e= \{0\}$. Then $0=A_ge\simeq A_g\otimes_R Re$. Also,  $\operatorname{Ann}(_RA_g)=R(1-\epsilon_g)$. Thus, the items (ii) and (v) of Proposition \ref{1pic} imply that $R\epsilon_g\simeq \operatorname{End}(_RA_g)$. Since $\operatorname{End}(_RA_g)\simeq A^*_g\otimes_RA_g$ we obtain that
$$0=A^*_g\otimes_RA_g\otimes_R Re\simeq R\epsilon_g\otimes_R Re=R\epsilon_ge,$$
which implies $\epsilon_ge=0$ and (i) $\Rightarrow$ (ii) is proved. For (ii) $\Rightarrow$ (iii), assume that $A_{g}e\neq  \{0\}$. From (iii) $\Rightarrow$ (i) follows $\epsilon_{g\m}e=e$. Using (i) $\Rightarrow$ (ii), it follows $A_{g\m}e\neq \{0\}$.
\end{proof}

Let $\mathcal{E}$ be the set given in \eqref{E}. Given a non-empty subset $\mathcal{X}\subset \mathcal{E}$ we set 
$$G_{\mathcal{X}}=\{g\in G\mid (\forall e_i\in \mathcal{X})\,\, (\exists e_j\in \mathcal{X}) \text{   such that   } A_g\otimes_R Re_i\simeq Re_j\}.$$
Observe that $G_{\mathcal{X}}$ is non-empty since $1_G\in G_{\mathcal{X}}$. Moreover, we will see below that $G_{\mathcal{X}}$ is a subgroup of $G$.  Then, we consider the epsilon-strongly $G_{\mathcal{X}}$-graded ring $$A_{G_{\mathcal{X}}}:=\mathlarger{\mathlarger{\oplus}}_{g\in G_{\mathcal{X}}}\,A_g$$ which is a subring of $A$. Now we will prove the main result of this subsection.

\begin{thm}\label{semicase}  Let $\mathcal{E}$ be the set given in \eqref{E}, $\mathcal{X}\subset \mathcal{E}$ be a non-empty subset, $A_{G_{\mathcal{X}}}$ be as above and $P$ in $\lmod{R}$. The following statements hold:
	\begin{enumerate}[\rm (i)] 
	\item If $P$ is a   $G$-invariant progenerator  in $\lmod{R}$  then $A$ is graded equivalent to the  partial crossed product $\operatorname{END}_A(A\otimes_{R} P)$. \vm\vm
    \item If $[P]\in \pic{R}$ and $1\leq i \leq n$ then either $P\otimes_R Re_{i}=0$ or $P\otimes_R Re_{i}\simeq Re_{j}$ for some $e_j\in \mathcal{E}$.   \vm\vm
    \item $G_{\mathcal{X}}$ is a subgroup of $G$. Moreover, if $Q_{\mathcal{X}}:=\mathlarger{\mathlarger{\oplus}}_{e_i\in \mathcal{X}}Re_i$ then  $ \operatorname{END}_{A_{G_{\mathcal{X}}}}(A_{G_{\mathcal{X}}}\otimes_R Q_{\mathcal{X}})$ is a partial crossed product.\vm \vm
    \item $G_{\mathcal{E}}=\{g\in G\mid \epsilon_ge_i=e_i,\, \forall e_i \in \mathcal{E}\}=\{g\in G\mid \epsilon_ge_i\neq 0,\, \forall e_i \in \mathcal{E}\}$.\vm\vm
    \item $A_{G_{\mathcal{E}}}$ is graded equivalent to the  partial crossed product $ \operatorname{END}_{A_{G_{\mathcal{E}}}}(A_{G_{\mathcal{E}}}\otimes_R Q_{\mathcal{E}})$, where $Q_{\mathcal{E}}=\mathlarger{\mathlarger{\oplus}}_{e_i\in \mathcal{E}}Re_i$.
	\end{enumerate}
\end{thm}

\begin{proof} (i) Let $P$ be a   $G$-invariant progenerator  in $\lmod{R}$. It follows from Theorem \ref{pcrp} and Proposition \ref{epcros0} that  $\operatorname{END}_A(A\otimes_{R} P) $ is an epsilon-crossed product. Hence,  by  Theorem 35 of \cite{NYOP} one has that  $\operatorname{END}_A(A\otimes_{R} P) $  is a partial crossed product.
Since $P$ is a progenerator in $\lmod{R}$ we get $A\otimes_RP$ is a progenerator in $\lmod{A}$. Thus $A$ is graded equivalent to $\operatorname{END}_A(A\otimes_{R} P) $.\smallbreak

\noindent (ii)  
Let $[P]\in \pic{R}$. Then $P$ is a left finitely generated and projective $R$-module . Assume that $P\otimes_R Re_{i}$ is nonzero.  We claim that $P\otimes_R Re_{i}$ is indecomposable. Suppose that $P\otimes_R Re_{i}=U\oplus V$, where $U,\,V$ are non-trivial $R$-submodules of $P\otimes_R Re_{i}$. Since $P^*\otimes_R P\simeq \operatorname{End}(_RP)$ as $R$-bimodules, it follows from Proposition \ref{1pic} (ii) that there exists a central idempontent $\tilde{e}$ of $R$ such that $\operatorname{Ann}(P_R)=R\tilde{e}$ and  $P^{*}\otimes_R P\simeq R(1-\tilde{e})$. Thus, \[(P^*\otimes _R U)\oplus (P^*\otimes _R V)\simeq P^*\otimes _RP\otimes_R Re_{i}\simeq Re_i(1-\tilde{e}),\]
as $R$-bimodules. If $e_i\tilde{e}=e_i$ we obtain that  $Re_i=(Re_i)\tilde{e}\subset R\tilde{e}=\operatorname{Ann}(P_R)$ and hence $P\otimes_R Re_{i}=0$ which is an absurd. Thus, $e_i\tilde{e}=0$ and $e_i(1-\tilde{e})=e_i$ because $e_i$ is a primitive idempotent of $R$.
Consequently 
$$Re_i=Re_i(1-\tilde{e})\simeq (P^*\otimes _R U)\oplus (P^*\otimes _R V).$$
Since $Re_i$ is a left indecomposable $R$-module,  we may suppose that $Re_i\simeq P^*\otimes _R U$ and $P^*\otimes _R V=0$. Hence $P\otimes_R Re_{i}\simeq  P\otimes_R P^*\otimes _R U\simeq R(1-\tilde{e})\otimes_{R} U\simeq (1-\tilde{e})U$, where $(1-\tilde{e})U=\{(1-\tilde{e})\cdot u\,:\,u\in U\}$ is an $R$-submodule of $U$. Thus, it follows from $P\otimes_R Re_{i}=U\oplus V$ that $P\otimes_R Re_{i}=U$ and $V=0$. Therefore $P\otimes_R Re_{i}$ is indecomposable and  there is $e_j\in \mathcal{E}$ such that $P\otimes_{R} Re_i\simeq Re_j$. \smallbreak

\noindent (iii) Notice that $1\in G_{\mathcal{X}}$ because $A_1=R$ and $R\otimes_R Re_i\simeq Re_i$, for all $e_i\in \mathcal{X}$. Let   $g,h\in G_{\mathcal{X}}$  and $e\in \mathcal{X}$. By definition, there exists $f\in \mathcal{X}$ such that $A_he\simeq A_h\otimes_R Re\simeq Rf$. In particular, $A_he\neq 0$. It follows from Lemma  \ref{ege} that  $A_{h}e\neq 0$.  Then,  $\epsilon_{h\m}e\neq 0$ and consequently $\epsilon_{h\m}e=e$. Hence
\begin{align*}
	A_{gh}\otimes_R Re\simeq A_{gh}e=A_{gh}\epsilon_{h\m}e \simeq A_{gh}\epsilon_{h\m}\otimes_R Re\stackrel{\eqref{eqrep23}}{\simeq} A_{g}A_{h}\otimes_R Re.
\end{align*}
Since $g\in G_{\mathcal{X}}$, there exists $\tilde{f}\in G_{\mathcal{X}}$ such that $A_g\otimes_R Rf\simeq R\tilde{f}$. Thus
\begin{equation}\label{agh}A_{gh}\otimes_R Re\simeq A_{g}A_{h}\otimes_R Re \simeq A_{g}\otimes_RA_h\otimes_R Re\simeq R\tilde{f},\end{equation}
and it follows that $gh\in G_{\mathcal{X}}$. In order to show that $G_{\mathcal{X}}$ is closed under inverses, we consider for each $g\in G_{\mathcal{X}}$ the map $\lambda_g: \mathcal{X}\to \mathcal{X}$ defined by $\lambda_g(e_i)=e_j$ if $A_g\otimes_R Re_i\simeq Re_j$.
We claim that $\lambda_g$ is injective (and consequently bijective). In fact, if $\lambda_g(e)=\lambda_g(e')$  then  $A_g\otimes_R Re\simeq A_g\otimes_R Re'$. By Proposition \ref{1pic}  (v), $[A_g]\in \pic{R}$. Notice that $\operatorname{Ann}(_RA_g)=R(1-\epsilon_g)$ and hence Proposition \ref{1pic}  (ii) implies that  $R\epsilon_g\simeq \operatorname{End}((A_g)_R)$. Using that $\ep_g e=e$ and that $\operatorname{End}((A_g)_R)\simeq A_g^*\otimes_R A_g$, we have  
\begin{align*} 
	Re=R\epsilon_ge\simeq R\epsilon_g\otimes_R Re   \simeq A_g^*\otimes_R A_g\otimes_RRe\simeq A_g^*\otimes_R A_g\otimes_RRe'\simeq Re',
\end{align*}
which implies $e=e'$. Hence, $\lambda_g$ is bijective. Thus, given  $e\in \mathcal{X},$  there is $f\in \mathcal{X}$ such that  $A_g\otimes_RRf\simeq Re$. Therefore
$$A_{g\m}\otimes_RRe\simeq A_{g\m}\otimes_R A_g\otimes_RRf\stackrel{\eqref{eqrep23}}{\simeq} R\epsilon_{g\m}f.$$
Assume that $\epsilon_{g\m}f=0$. Then $af=a\epsilon_{g\m}f=0$, for all $a\in A_g$. Thus $A_gf=0$ which is an absurd because $A_gf\simeq Re\neq 0$. Consequently, $\epsilon_{g\m}f=f$ and $g\m\in G_{\mathcal{X}}$.

In order to check the second assertion, notice that $A_{G_{\mathcal{X}}}$   is  an epsilon-strongly graded ring. Observe that $A_g\otimes_{R} Q_{\mathcal{X}}\simeq Q_{\mathcal{X}}$, for all $g\in G_{\mathcal{X}}$. Then $Q_{\mathcal{X}}$ is  $G_{\mathcal{X}}$-invariant and the result follows from Proposition  \ref{epcros0}. \smallbreak

\noindent (iv) Let $g\in G_{\mathcal{E}}$. Then $A_ge\simeq A_g\otimes_R Re\neq 0$, for all $e\in  \mathcal{E}$. It follows from Lemma \ref{ege} $\epsilon_ge=e$. For the reverse inclusion, assume that $g\in G$ and $\epsilon_ge=e$, for all $e\in \mathcal{E}$. Again, by Lemma \ref{ege}, we obtain that $A_g\otimes Re\simeq A_ge\neq 0$. By (ii) and Proposition \ref{1pic} (v), we have that $A_g\otimes Re\simeq Rf$, for some $f\in \mathcal{E}$. Thus, $g\in G_{\mathcal{E}}$. The equality on the right side of (iv) is immediate.\smallbreak

\noindent (v) It follows from \cite[(18.10)(C)]{L}  that $Q_{\mathcal{E}}=\mathlarger{\mathlarger{\oplus}}_{e_i\in \mathcal{E}}Re_i$ is a progenerator in $\lmod{R}$. Also, it is clear that $Q_{\mathcal{E}}$ is $G_{\mathcal{E}}$-invariant. Thus the result follows from (i).
\end{proof}

Let $\mathbb{I}=\{1,2,..., n\}$. For each $g\in G$, let $\mathbb{I}_g=\{i\in \mathbb{I}\,:\, \exists j\in \mathbb{I}\text{ with } A_g\otimes_R Re_i\simeq Re_j\}$ and $\alpha_g: \mathbb{I}_{g\m}\to \mathbb{I}_g,$ where  $\alpha_g(i)=j$ if and only if  $A_{g\m}\otimes_R Re_i\simeq Re_{j}$. We saw in the proof of Theorem \ref{semicase} (iii) that $A_{g\m}\otimes Re_i\simeq Re_j$ implies that $A_{g}\otimes Re_j\simeq Re_i$. Hence, $\alpha_g$ is well-defined and it is a bijection with inverse  $\alpha_{{g}\m}$.

\begin{cor}\label{corpar} The following assertions hold. 
	\begin{enumerate}[\rm (i)] 
	\item If $A$ is strongly graded then it is equivalent to a crossed product.
	\item The family $\alpha=(\{\mathbb{I}_g\}_{g\in G}, \{\alpha_g\}_{g\in G})$ determines a set-theoretic  partial action of $G$ on $\mathbb{I}$ such that $\mathbb{I}_{g\m}= \mathbb{I}_g,$ for any $g\in G.$  
	\end{enumerate}
\end{cor}

\begin{proof}
\noindent (i) Because $A$ is strongly graded, we have that $\epsilon_g=1_A$ for all $g\in G$. Then $A_{G_{\mathcal{E}}}=A$ and it follows from Theorem \ref{semicase} (v) that $A$ is graded equivalent to the epsilon-crossed product $ C=\operatorname{END}_{A}(A\otimes_R Q_{\mathcal{E}})$. But,  $A$  is strongly graded and then $C$ is strongly graded. Hence $C$ is a crossed product.\smallbreak

\noindent (ii) 	Clearly $\mathbb{I}_e=\mathbb{I}$ and $\alpha_e=\id_{\mathbb{I}}$. As we saw above, $\alpha_g$ is invertible and $\alpha^{\m}_g=\alpha_{g\m}$. In order to verify that $\alpha_{gh}$ extends  $\alpha_{g}\circ \alpha_{h}$ consider $i\in \mathbb{I}_{h\m}$ such that $j=\alpha_{h}(i)\in \mathbb{I}_{g\m}$. Then $A_{h\m}\otimes Re_i\simeq Re_{j}$ and $A_{g\m}\otimes Re_{j}\simeq Re_{k}$. Thus, $\alpha_{g}(j)=k$. As in \eqref{agh} we have that  
$A_{h\m g\m}\otimes_R Re_i\simeq A_{h\m}\otimes_RA_{g\m}\otimes_R Re_i$ which implies $\alpha_{gh}(i)=k$. Therefore $\alpha$ is a partial action of $G$ on $\mathbb{I}$. Finally, if $i\in \mathbb{I}_{g\m}$ then $A_{g\m}e_i\simeq A_{g\m}\otimes_R Re_i\simeq Re_{j}\neq 0$, for some $j\in \mathbb{I}$. By Lemma \ref{ege}, $A_{g}\otimes_R Re_i\simeq A_{g}e_i\neq 0$. From Proposition \ref{1pic} (v) and Theorem \ref{semicase} (ii) follow that $A_{g}\otimes_R Re_i\simeq Re_k$, for some $k\in \mathbb{I}$. Therefore, $i\in \mathbb{I}_g$. Since $g$ is arbitrary it follows that $\mathbb{I}_{g\m}=\mathbb{I}_g$.
\end{proof}

\begin{rem} Let $G_{\mathcal E}$ be the subgroup of $G$ defined in (iv) of Theorem \ref{semicase}. Denote by $\alpha_{G_{\mathcal E}}=(\{\mathbb{I}_g\}_{g\in G}, \{\alpha_g\})_{g\in G_{\mathcal E}}$ the restriction of the partial action $\alpha$ of $G$ on $\mathbb{I}$ which was defined above of Corollary \ref{corpar}. Observe that $\alpha_{G_{\mathcal E}}$ is indeed a global action of $G_{\mathcal E} $.
 \end{rem}

In order to illustrate Theorem \ref{semicase}, we consider the following examples.

\begin{exa} \label{ejem2.2}
	Let  $S$ be a commutative ring with a non-zero identity $1_S$ and $I$ a unital ideal of $S$ such that  $1_I\neq 1_S$. Consider the following sets of matrices:
	\begin{displaymath}
		A=  \left( 
		\begin{matrix}
			S & S  \\
			I &  S
		\end{matrix}
		\right),
		\quad
		R=A_0 =  \left( 
		\begin{matrix}
			S &  0 \\
			0 & S
		\end{matrix}
		\right)
		\,\, \text{  and  } \,\,
		A_1 =  
		\left( 
		\begin{matrix}
			0 & S  \\
			I & 0
		\end{matrix}
		\right).
	\end{displaymath}
	It is clear that  $A=A_0\oplus A_1$  is an epsilon-strongly  $\mathbb{Z}_2$-graded ring. Also, notice that  $\epsilon_{0}=\operatorname{diag}(1_S,1_S)$ and $\ep_{1}=\operatorname{diag}(1_S, 1_I)$. It is straightforward to check that there is no epsilon-invertible element in $A_1$. Thus, $A$  is not an epsilon-crossed product. Suppose that $S:=\Bbbk\times \Bbbk$ and $I=\Bbbk\times \{0\}$, where $\Bbbk$ is a field. Then $R\simeq S\times S=\Bbbk^4$ is semiperfect. Moreover,  the canonical basis $\{e_j\,:\,j=1,\ldots,4\}$ of $\Bbbk^4$ is  a complete orthogonal set of primitive idempotents and $Re_i\simeq Re_j$, for all $i,j=1,\ldots,4$. Then $\mathcal{E}=\{e_1\}$ and $G_{\mathcal{E}}=\mathbb{Z}_2$. Let $Q_{\mathcal{E}}=Re_1$. By Theorem \ref{semicase} (v), we have that $A$ is graded equivalent to the epsilon-crossed product $C=\operatorname{END}_{A}(A\otimes_R Re_1)$. 
\end{exa}

\begin{exa} Let $A=\mathlarger{\mathlarger{\oplus}}_{g\in G}\,A_g$ be an epsilon-strongly graded ring   with $R=A_1$ and consider $B:=\operatorname{M}_n(A)$. By Example \ref{exam}, $B$ is epsilon-strongly graded. If $R$ is semiperfect then $B_1=\operatorname{M}_n(R)$ is also semiperfect (see for instance \cite[(23.9)]{L2}).  We will check that  if $A$ is is graded equivalent to an partial-crossed product, then so is $B.$ Let $\mathcal{E}_R=\{e_1,\ldots,e_m\}$  be a set of
	pairwise orthogonal primitive idempotents such that $Re_1,\ldots,Re_m$ is a complete irredundant set of representatives of the indecomposable projective left $R$-modules. Suppose that $\epsilon_ge=e,$ for all $e\in \mathcal{E}_R $ and $g\in G$. 
By Morita equivalence, we have that $\mathcal{E}_B=\{Ie_1,\ldots,Ie_m\}$, where $I$ is the identity matrix of $B$,  is a set of
pairwise orthogonal primitive idempotents such that $M_n(R)Ie_1,\ldots,M_n(R)Ie_m$ is a complete irredundant set of representatives of the indecomposable projective left $B_1$-modules.
Let $E_g=\epsilon_gI$ be defined in \eqref{Ep}. Then, for all $Ie\in \mathcal{E}_B$, we have that $E_gIe=I\epsilon_ge=Ie$. Thus, $G_{\mathcal{E}_B}=G$ and $B_{G_{\mathcal{E}_B}}=B_G=B$. It follows from Theorem \ref{semicase} (v) that  $B$ is graded equivalent to an epsilon-crossed product.
\end{exa}

\end{document}